\documentclass[12pt]{iopart}

\usepackage{iopams}
\usepackage[font=footnotesize,labelfont=bf]{caption}
\usepackage{graphicx}
\usepackage{subfig}
\usepackage[breaklinks=true,colorlinks=true,linkcolor=blue,urlcolor=blue,citecolor=blue]{hyperref}

\expandafter\let\csname equation*\endcsname\relax

\expandafter\let\csname endequation*\endcsname\relax

\usepackage{amsmath}
\usepackage{amsthm}
\usepackage{bm}
\usepackage{bbm}
\usepackage{siunitx}
\usepackage{float}
\usepackage{setspace}
\usepackage{bigints}
\usepackage{algorithm}
\usepackage{algpseudocode}
\usepackage{tikz}
\usetikzlibrary{spy}
\usepackage{epstopdf}
\usepackage{dsfont}

\def\uu{\text{\bf u}}
\def\ff{\text{\bf f}}
\def\ww{\text{\bf w}}
\def\sfA{{\sf A}}
\def\AA{\text{\bf A}}
\def\sfK{\text{\sf K}}
\def\KK{\text{\bf K}}

\def\CC{\mathbb{C}}
\def\RR{\mathbb{R}}
\def\cH{\mathcal{H}}

\def\dist{\text{\rm dist}}
\def\supp{\text{\rm supp}}

\newtheorem{theorem}{Theorem}

\newtheorem{proposition}{Proposition}
\newtheorem{definition}{Definition}

\begin{document}

\title[Photoacoustic tomography with partial data]{Photoacoustic tomography in attenuating media with partial data}

\author{Benjamin Palacios}

\address{Departments of Mathematics, Pontificia Universidad Cat\'olica de Chile.\\ Av. Vicu\~na Mackenna 4860, Macul, Santiago, Chile}
\ead{benjamin.palacios@mat.uc.cl}
\vspace{10pt}

\begin{abstract}
The attenuation of ultrasound waves in photoacoustic and thermoacoustic imaging presents an important drawback in the applicability of these modalities. This issue has been addressed previously in the applied and theoretical literature, and some advances have been made on the topic. In particular, stability inequalities have been proposed for the inverse problem of initial source recovery with partial observations under the assumption of unique determination of the initial pressure. The main goal of this work is to fill this gap, this is, we prove the uniqueness property for the inverse problem and establish the associated stability estimates as well. The problem of reconstructing the initial condition of acoustic waves in the complete-data setting is revisited and a new Neumann series reconstruction formula is obtained for the case of partial observations in a semi-bounded geometry. A numerical simulation is also included to test the method.
\end{abstract}

\section{Introduction}

The inverse problem of photoacoustic tomography (PAT) consists in the determination of the initial source of acoustic waves from measurements acquired at the boundary of a domain $\Omega\subset\RR^3$. This problem has been extensively studied for several years from theoretical and applied perspectives, leading to great developments in the imaging modality of PAT. For media without acoustic attenuation the inverse problem relies in the analysis of the classical wave operator $\partial_t^2 - c^2\Delta$, where $c(x)$ is a smoothly varying function representing the speed of sound in biological tissues. In this setting (and more generally for a Laplace-Beltrami operator $\Delta_g$), uniqueness for the inverse problem was established by Stefanov and Uhlmann in \cite{SU} for the general case of observations taking place at a portion of the boundary ---the partial data problem. In the same article, precise conditions for the observation time and the underlying geometry were determined in order to guarantee uniqueness, stability, and reconstruction of the initial source. Particularly, the optimal conditions guarantying  injectivity of the forward operator lean on a powerful unique continuation property for second order hyperbolic operators, consequence of Holmgren's uniqueness result (generalized later by Fritz John) 
for operators with analytic coefficient in a Euclidean background (see \cite[\S 8]{HoI}) and later extended to more general wave operators by other authors. The next formulation is derived from the work of Tataru \cite{T3}. 
\begin{theorem}{\cite[Theorem 4]{SU}}\label{thm:John}
Assume that $u\in H^1_{loc}$ satisfies $(\partial^2_t-c^2\Delta)u=0$ and $u=0$ in a neighborhood of $[-T,T]\times\{x_0\}$, for some $T>0$ and $x_0\in\RR^n$. Then, 
\[
u(t,x) = 0\quad\text{for}\quad |t|+\dist(x_0,x)\leq T.
\]
\end{theorem}
Throughout the paper we denote by $\dist(x,y)$ the distance function induced by a metric $c^{-2}dx^2$, this is, the infimum of the length of piece-wise $C^1$-curve segments joining $x$ and $y$.

This theorem holds, in fact, for more general second order hyperbolic operators with smooth coefficients and analytic in time. 
In order for this property to be applied in the photoacoustic problem
it is necessary to extend the equation to negative times by considering even extensions of the wave fields. This is possible in part due to the absence of attenuation ---commonly modeled by a damping term $a(x)\partial_tu$ with $a\geq 0$--- and also because of the specific form of the initial conditions, $(u,u_t)|_{t=0}=(f,0)$. The difficulty of using the previous theorem in the attenuating media case (e.g. for the damped wave equation) and partial data was first noticed in \cite{H}. In contrast, for the complete data case it is still feasible to deduce uniqueness from Theorem \ref{thm:John}, but this requires to double the observation time needed in the unattenuated case which of course is  not a sharp requirement (see Theorem 3.1 in the aforementioned paper). An improvement in the complete data case was given in \cite{AP} by means of a foliation condition, demanding the geometry to be such that it allows the existence of particular families of strictly convex hypersurfaces with dense union. In that paper, the authors found that assuming the level sets of a boundary defining function are all strictly convex and foliate the domain, uniqueness with complete data holds for the sharp observation time and this still holds true by adding an integro-differential attenuation term of memory-type. The foliation condition is a common assumption on inverse geometric problems and it has been used, for instance, in the problem of determining a sound speed from the knowledge of photoacoustic data and the initial condition \cite{SU2}. In regard of uniqueness for the inverse photoacoustic problem, the foliation assumption appearing in \cite{AP} seems to be just a technical tool, allowing the application of a layer stripping argument from the boundary towards the interior, and thus by-passing the lack of a more general unique continuation property along the lines of Theorem \ref{thm:John}. 

On the other hand, the stability for the damped PAT problem with partial observations has been addressed previously in \cite{H} and more recently in \cite{HaN}, where a reconstruction procedures was also presented based on the study of the adjoint problem and iterative methods. In both articles, the stability is derived subject to the hypothesis of injectivity, however the injectivity of the partial data problem is left as an open question.

The main goal of this paper is to establish uniqueness and stability for the damped wave equation from observation taken in a relative open subset of the boundary $\Gamma\subset\partial\Omega$. This is the content of Theorem \ref{thm:uniq_stab}. Instead of using a specific unique continuation results for the wave equation, as the likes of Theorem \ref{thm:John}, the time-independency of the coefficients will permit us to deduce injectivity by means of Riesz's Lemma and uniqueness of the Cauchy problem for elliptic operators. The idea is to first obtain a microlocal stability estimate based on a visibility condition of singularities, then use this inequality to derive the injectivity with the aid of the aforementioned tools, and subsequently deduce stability by means of a compactness-uniqueness argument. This methodology is taken from \cite{BLR}. Being a method that exploits uniqueness properties of elliptic equations, it is not specifically tailored to the wave operator, and therefore, it is not expected to give the optimal restrictions on the parameters of the problem. Indeed, for the unattenuated case, it imposes a coarser lower bound on the observation time ---the optimal one obtained from 
Tataru's result for time-analytic coefficients.

We also consider the reconstruction question on both, the complete and partial data settings. In the former, this implies revisiting the reconstruction procedure based on time-reversal for the damped wave equation with arbitrary smooth bounded attenuations, introduced by the author in \cite{P}. We amend a few computations carried out in that first publication, and by doing so, we are naturally led to a generalization of the original result. The new version of the reconstruction is given in Theorem \ref{thm:reconstruction_freespace}. Regarding partial data, we first look at the classical setting of a bounded region of interest with acoustically transparent boundary where we make the connection between this article and the results in \cite{HaN}, and on the other hand, we consider the setting of an unbounded geometry with partially reflecting (or dissipative) boundary, where we obtain the reconstruction formula in Theorem \ref{thm:reconstruction}. Numerical simulations were performed to visualize the reconstruction procedure of the latter theorem and whose results are presented at the end of this article.

\section{Inverse problem and main result}\label{sec:IP}

Let $\Omega\subset\RR^n$ be an open connected set with smooth boundary, not necessarily bounded, and let $c(x)\geq c_0>0$ be a sound speed in $\RR^n$. We assume $c=1$ for all $x \notin \Omega$ and consider an attenuation coefficient $a(x)\geq 0$ compactly supported and bounded in $\Omega$. We consider the following damped wave operator
\begin{equation}\label{opP}
\square_au := \partial^2_tu -c^{2}(x)\Delta_g u+ a(x)\partial_t u,
\end{equation}
where we assume the sound speed $c$, the underlying metric $g$ and the attenuation coefficient $a$ are known. Furthermore, and for simplicity in the exposition, we make the following two assumptions: we suppose $c$ and $a$ are smooth functions; and $\Delta_g=\Delta$, this is we assume $g$ is the Euclidean metric. The subsequent analysis can be carried out under less regularity and for more general Riemannian metrics. We do not pursue the question of optimality in the regularity of these coefficients. 

The sound speed function $c(x)$ induces the {\em sound speed metric} $c^{-2}(x)dx^2$. Below we consider geodesics $\gamma$ associated to it and the convexity of the boundary (or portions of it) is understood in terms of this metric as well.\\

For a given initial source $\ff = (f_1,f_2)$ we consider $u$ to be the solution of
\begin{equation}\label{PAT}
\left\{
\begin{aligned}
u_{tt}-c^{2}(x)\Delta u + a(x)\partial_tu&=0,&&\text{in }[0,T]\times U,\\
(u,u_t)|_{t=0} &= \ff,&&\text{in }U,\\
\end{aligned}
\right.
\end{equation}
with either $U=\RR^n$ or $U=\Omega$. In the former case, we assume $\Omega\Subset \RR^n$ is a bounded region with an acoustically transparent boundary and we call this the {\em transparent-boundary} setting. In the later case, $U=\Omega$ is taken to be an unbounded set with boundary, provided with a smooth compactly supported and nonnegative (and non-null) boundary function $\lambda$, and dissipative (Robin) boundary conditions:
\begin{equation}\label{PAT_bc}
Bu:= \partial_\nu u + \lambda \partial_tu = 0,\quad\text{on }[0,T]\times\partial\Omega.
\end{equation}
For the initial boundary value system \eqref{PAT}-\eqref{PAT_bc}, we say $\Omega$ has a {\em partially reflecting boundary}.

The forward well-posedness of this system follows from the discussion in \cite{BLR} and finite propagation speed, which allows us to restrict the problem to a bounded region. Indeed, for $\ff\in H^1(U)\times L^2(U)$ compactly supported inside $U$, solutions $\uu(t) = (u(t),u_t(t))$ belong to the energy space $H^1(U)\times L^2(U)$ for all $t\in[0,T]$. Furthermore, due to $U$ being unbounded and $\ff$ compactly supported, time-independent functions are not solutions to \eqref{PAT}. 

Let's now define the (partial data) observation map. For this, let $\Gamma$ be a relatively open and strictly convex (for the sound speed metric $c^{-2}(x)dx^2$) subset of the boundary. For a fixed positive and continuous function $s(x)$ in $\Gamma$ we call the {\em observation set} to 
\begin{equation}\label{def:obs_set}
\mathcal{G}:=\{(t,x): x\in\Gamma,\;0<t<s(x)\}.
\end{equation}
The {\em observation map} is the operator mapping initial sources $\ff$ to Dirichlet boundary measurements 
\begin{equation}\label{def:obs_map}
\Lambda_a \ff := u|_{\mathcal{G}}.
\end{equation}
It is a common assumption to take $s(x) \equiv T$, in which case $T$ is called the {\em observation time} and then $\mathcal{G} = (0,T)\times\Gamma$. For simplicity, we will restrict our analysis to this case. We say we have {\em complete data} whenever $\Gamma = \partial\Omega$, and {\em partial data} otherwise. 
In addition, in the case of $U=\Omega$ in \eqref{PAT}, we fix $\Gamma := \{\lambda>0\}$, and for a given $\lambda_0>0$ we also define $\Gamma_0:=\{\lambda>\lambda_0\}$. 

The reader might have noticed that we are interested in two different geometries for the photoacoustic problem. As a matter of fact, we will address the complete and partial data problems for the case of $U=\RR^n$ and $\Omega$ bounded with acoustically transparent boundary, which will be referred to as the setting of {\em bounded geometry with transparent-boundary}; and secondly, we will consider the partial data problem for an unbounded $U=\Omega$ with boundary, denoted as the setting of {\em semi-bounded geometry with partially-reflecting boundary}.\\

The inverse photoacoustic tomography problem in acoustically damping media consists in determining an initial source of the form $\ff=(f,-af)$ from the knowledge of $\Lambda_a \ff$. See  \cite{H}. As we mentioned above, the main differences with the classical formulation of the problem is the presence of the damping coefficient $a(x)$ which breaks the time-symmetry of the unattenuated case by impeding a $C^1$-extension to negative times. Some already classical references on the theory of the unattenuated photoacoustic problem are \cite{FPR, SU,SU1} and the survey \cite{SUsurvey}; regarding the damping media case previous work include \cite{H,P,AP,HaN}. The inverse problems for partially and perfectly reflecting boundaries are addressed  in \cite{AM,SY,NK,CO}. Our particular choice of geometry is then motivated by \cite{AM} and \cite{CO}, where partially reflecting boundaries are considered in the former, while the semi-bounded geometry is employed in the latter article. We refer the reader to those papers and reference therein for details on the implementations of these photoacoustic settings in the practice.

In the context of damping media, in both \cite{H,HaN}, the issue of stability in the partial data setting is investigated. 
It is stated there that under a visibility condition for singularities of the initial source (see definition below) Lipschitz stability of the inverse problem holds as long as $\Lambda_a$ is an injective map. Nevertheless, the injectivity of the partial data measurement operator is not proven. This work intends to bridge the gap and complete the treatment of the damped PAT problem by establishing the injectivity and stability for the partial data case. 

Naturally, at the center of the derivation of stability (and uniqueness) lies the microlocal assumption of observability of singularities generated at the initial time. This is the visibility condition for singularities, or equivalently, the {\em geometric control condition} of \cite{BLR}. We state below a particular instance of such condition in the case of a strictly convex observation surface with a transparent boundary.
\begin{definition}\label{def:vis_cond1}
Let $T>0$, $\Omega_0\Subset\Omega$ a bounded subdomain, and $\Gamma\subseteq\partial\Omega$ strictly convex (for the metric $c^{-2}(x)dx^2$). 
We say the {\em visibility condition} is satisfied by $(\Gamma, T,\Omega_0)$ if for every unit speed geodesic $\gamma(t)$, with $\gamma(0)\in\bar{\Omega}_0$, there exists $|t|<T$ such that $\gamma(t)\in\Gamma$.
\end{definition}

For the case of a partially-reflecting boundary, the visibility condition needs to be adapted to account for the reflections of the singularities at the boundary. We consider instead broken-geodesics whose trajectories and reflections near the boundary are determined by the laws of geometric optics. The condition in this case reads:

\begin{definition}\label{def:vis_cond2}
Let $T>0$, $\Omega_0\Subset\Omega$ a bounded subdomain, and $\Gamma\subset\partial\Omega$ strictly convex. 
We say the {\em visibility condition} is satisfied by $(\Gamma, T,\Omega_0)$ if for every unit speed broken-geodesic $\gamma(t)$, with $\gamma(0)\in\bar{\Omega}_0$, there exists $|t|<T$ such that $\gamma(t)\in\Gamma$.
\end{definition}

We will refer to both of the previous definitions as the visibility condition, and the specific choice will become clear from the assumptions on the geometry of the problem.\\

The main result of this paper is the injectivity and stability of $\Lambda_a$ under the visibility condition. The definition of the energy space $\mathcal{H}(\Omega_0)$ is given in the next Section \ref{sec:preliminaries}.

\begin{theorem}\label{thm:uniq_stab}
Let's assume the visibility condition holds for some $T>0$, $\Omega_0\Subset\Omega$ with smooth boundary, and $\Gamma$ strictly convex. Then, the observation map $\Lambda_a$ is injective, this is, whenever $\Lambda_a\ff = 0$ for some $\ff \in \cH(\Omega_0)$, then $\ff\equiv 0$. Moreover, the next stability estimate holds:
$$
\|\ff\|_{\cH(\Omega_0)}\leq C\|\Lambda_a\ff\|_{H^1((0,T)\times \Gamma)},\quad\forall\; \ff\in \cH(\Omega_0).
$$
In the particular case of $\Omega$ bounded with transparent-boundary, one can set $\Omega_0 = \Omega$ provided the visibility condition holds throughout the whole domain.
\end{theorem}

It is not clear that this injectivity result is the optimal for the damped PAT problem. By comparing it with the unattenuated case, the smallest time needed for the injectivity property of Theorem \ref{thm:uniq_stab} to hold is in general larger than the optimal injectivity time for the unattenuated case. The former is given by $T_1(\Omega_0,\Gamma)$, which is the minimum time necessary so that for all $(x,\xi)\in S^*\overline{\Omega}_0$ (the unit cosphere bundle), at least one of $\gamma_{x,\xi}$ and $\gamma_{x,-\xi}$ reaches $\Gamma$ before time $T_1(\Omega_0,\Gamma)$; the latter time is defined as $T_0(\Omega_0,\Gamma) := \max\{\dist(x,\Gamma):x\in\Omega_0\}$. Then $T_1(\Omega_0,\Gamma)\geq T_0(\Omega_0,\Gamma)$, and Theorem \ref{thm:uniq_stab} holds for $T>T_1(\Omega_0,\Gamma)$. 

The optimal lower bound (i.e. $T>T_0(\Omega,\partial\Omega)$) was proven in \cite{AP} for the complete data problem (thus $\Omega$ bounded), assuming the existence of a particular foliation of $\Omega$ by smooth and strictly convex hypersurfaces. Here we do not assume a foliation condition and we are able to improve the lower bound obtained originally in \cite{H} for observation in the whole boundary, which is $2T_0(\Omega,\partial\Omega)$, thus requiring now $T>\min\{2T_0(\Omega,\partial\Omega), T_1(\Omega,\partial\Omega)\}$. 

Notice that, since $T_1$ depends on the geodesics of the sound speed metric while $T_0$ is defined in terms of $C^1$-curve segments, it could happen that $T_1$ is significantly larger that $T_0$ and in particular $T_1>2T_0$. An extreme case is when $(\Omega,c^{-2}dx^2)$ is a trapping manifold, hence $T_1=\infty$. Nevertheless, for some metrics (for instance those close to the Euclidean one) $T_1$ could still be smaller than $2T_0$. This occurs for example if $\Omega=B(r,0)$, the ball fo radius $r>0$ and center at the origin, and $c(x)\approx 1$. Then, $T_0\approx T_1\approx r$, and our result provides a significant improvement on the lower bound for $T$ guaranteeing uniqueness of the inverse problem.\\


Recently, Stefanov \cite{St} demonstrated that the foliation condition allows the construction of a pseudo-convex function from where it is possible to deduce conditional H\"older stability estimates for the partial data problem, even in the case of initial conditions with support not completely contained inside the visible region ---this latter region defined as the subdomain satisfying the geometric control condition. However, the portion of the domain that is H\"older stably recoverable satisfies (a-posteriori) that singularities issued from there are visible in the measurements, which is a consequence of the foliation.

\section{Preliminaries and organization of the paper}\label{sec:preliminaries}
Let's introduce some functional spaces that will be relevant in the subsequent analysis. 
We denote by $H_D(\Omega)$ the completion of $C^\infty_0(\Omega)$ under the Dirichlet norm
$$\|f\|_{H_D(\Omega)}^2 := \int_{\Omega}|\nabla f|^2dx,$$
which is topologically equivalent to $H^1_0(\Omega)$. We similarly define $H_D(\Omega_0)$ for some bounded subdomain $\Omega_0\Subset\Omega$ with smooth boundary, which is identified as 
the subspace of $H_D(\Omega)$ containing all functions that vanish outside $\Omega_0$.

Throughout the paper the spaces $L^2(\Omega)$ and $L^2(\Omega_0)$ will stand for the set of all square integrable  functions for the sound speed measure $c^{-2}(x)dx$ in their respective sets. We then define the energy space of initial sources as
$$
\mathcal{H}(\Omega_0):= H_D(\Omega_0)\times L^2(\Omega_0).
$$
Similarly as the case of $H_D(\Omega)$, the energy space $\mathcal{H}(\Omega_0)$ is topologically equivalent to $H^1_0(\Omega_0)\times L^2(\Omega_0)$. 

For a vector valued function $\ff = (f_1,f_2)$ we set $\Pi_j\ff := f_j$, $j=1,2$, the projections to the first and second components, respectively. We also define the orthogonal projection operator $\Pi_{\Omega_0}$ as the map that assigns $\Pi_{\Omega_0}f:=g$, where
$$
\Delta g = \Delta f\quad\text{in }\Omega_0,\quad g|_{\partial\Omega} = 0.
$$
Hence, $\Pi_{\Omega_0}:H_D(\Omega)\to H_D(\Omega_0)$ continuously. 
More details on this can be found in \cite[\S2.2]{SY}. Finally, for a function $\ff = (f_1,f_2)\in \mathcal{H}(\Omega)$ we write ${\bf\Pi}_{\Omega_0}\ff := (\Pi_{\Omega_0}f_1,\mathds{1}_{\Omega_0}f_2)\in \mathcal{H}(\Omega_0)$, where $\mathds{1}_{\Omega_0}$ stands for the characteristic function of ${\Omega_0}$.

Recall that denoting $\uu = (u,u_t)$ it is possible to write \eqref{PAT} in the form
$$\uu_t = \text{\bf P}_a\uu\quad \text{for}\quad\text{\bf P}_a := \left(\begin{matrix} 0&I\\c^2\Delta& -a\end{matrix}\right),$$
with the operator $\text{\bf P}_a$ (augmented with Robin boundary conditions) defining a strongly continuous semigroup $e^{t{\bf P}_a}$ which given an initial data $\ff = (f_1,f_2)$ it assigns the solution to the previous system $\uu(t) = e^{t\text{\bf P}_a}\ff$ at time $t>0$ (see, for instance, \cite{Fu}). 

The energy functional associated to the damped wave equation in a region $\Omega$ and time $t>0$ is given by
$$
E_{\Omega}(\uu(t)):=\|\uu(t)\|^2_{\cH(\Omega)} = \|u(t)\|^2_{H_D(\Omega)} + \|u_t(t)\|^2_{L^2(\Omega)},
$$
while its extended energy functional is defined as
$$
\mathcal{E}_{\Omega,t}(\uu) := E_{\Omega}(\uu(t)) + 2\int^t_0\int_{\Omega}ac^{-2}|u_t(s)|^2dsdx.
$$
This second energy functional takes into account the portion of the energy lost due to inner attenuation. In a closed system (i.e.  without dissipation of energy through the boundary of $\Omega$) this quantity is conserved. In the case of system \eqref{PAT}-\eqref{PAT_bc}, it is not hard to verify that the extended energy functional decreases in time as a consequence of the dissipative boundary conditions.

The rest of the paper is organized as follows. We start by revisiting the time-reversal operators introduced in \cite{H} and \cite{P} for the damped wave equation. 
We also present a standard geometric optic construction of microlocal approximate solutions (parametrix), and as a consequence of such analysis, we obtain continuity estimates for the observation map and the ellipticity of a pseudodifferential operator involving a microlocal back-propagation of the boundary data. All of this is later used to prove our main result, Theorem \ref{thm:uniq_stab}, whose proof is divided into two parts corresponding respectively to the bounded geometry case in Section \ref{subsec:uniq_bdd}, and the unbounded geometry in Section \ref{subsec:uniq_unbdd}. In Sections \ref{subsec:rec_bdd} and \ref{subsec:rec_unbdd}, the approximate solutions serve us to analyze the iterative time-reversal-based reconstruction of \cite{P}, for the respective cases of bounded geometries with transparent-boundary and for semi-bounded geometries with partially reflecting boundary. We conclude the article with Section \ref{sec:numerics}, where we illustrate the reconstruction procedure in the semi-bounded geometry setting with numerical simulations.

\section{Time-reversal and geometric optics}\label{sec:TR}

In this section we review two time-reversed systems available in the literature in the context of acoustically damping media. They generate iterative reconstruction procedures based on Neumann series. We start with the direct extension of the original sharp time-reversal method of \cite{SU}, and whose respective convergence of the associated Neumann series was discussed in \cite{H} under a smallness condition assumed over the damping coefficient. The second method, introduced in \cite{P} and later extended in \cite{AP} to integro-differential attenuations, provides a modification of the former and gives the convergence of another Neumann series, this time independently of the size of the attenuation. As expected, the speed of convergence decreases with the amplitude of the damping coefficient. Although the aforementioned time-reversal systems where originally defined in a free space setting, our presentation below is given in a more general fashion to include the case of an unbounded domain with a smooth partially reflecting boundary, this is, involving dissipative boundary conditions of Robin type (see \cite{AM} for the case of complete data in a bounded geometry).

\subsection{Time-reversal operators}
Let's set $\ff = (f_1,f_2)\in \mathcal{H}(\Omega_0)$, and let $u$ be the solution to the initial boundary value problem \eqref{PAT}; augmented with boundary condition \eqref{PAT_bc} when $U=\Omega$.
We recall the definition $\Lambda_a\ff := u|_{(0,T)\times\Gamma}$ for the measurement operator
which according to \cite[Theorem 5.5]{BLR} and trace inequalities \cite{Evans} (see also \cite{T2} for sharper results) it maps continuously
$$\Lambda_a:\mathcal{H}(\Omega_0)\longrightarrow C([0,T];H^{1/2}(\Gamma)).$$ 
Other continuity results can be obtained by studying $\Lambda_a$ in the context of {\em Fourier Integral Operators} (FIO's) in a similar fashion as what we do in the next subsection.

We remark that for $U=\Omega$ unbounded, finite propagation speed allows us to restrict \eqref{PAT} to a sufficiently large bounded domain $\Omega'\subset\Omega$ with smooth boundary and such that $\Gamma\subset\partial\Omega'\cap\partial\Omega$. If $\partial\Omega$ is partially reflecting then we impose $\partial_\nu u=0$ on $\partial\Omega'\backslash\Gamma$. 

For the boundary observations $h=\Lambda_a\ff$ we define the (standard or non-dissipative) time-reversal operator as follows. Let's set
$$
\begin{aligned}
&B'(v,h):=v-h &&\text{ for transparent-boundaries,}\quad\text{and}\\
&B'(v,h):=\partial_\nu v + \lambda\partial_th &&\text{ for partially reflecting-boundaries.}
\end{aligned}
$$
Then, the time-reversal operator $\sfA_a$ is the map that takes $h \to (v,v_t)|_{t=0}$ with $v$ solution to the backward system
\begin{equation}\label{Homan_TR}
\left\{\begin{aligned}
(\partial^2_t + a\partial_t- c^2\Delta)v &= 0,&& \text{in }(0,T)\times\Omega,\\
B'(v,h) &= 0 ,&&\text{on }(0,T)\times\partial\Omega,\\
(v,v_t)|_{t=T}&= \left(\mathcal{P}(h(T)),0\right),&&\text{in }\Omega.\\
\end{aligned}
\right.
\end{equation}
Here we set
$$
\mathcal{P}(h(T)) :=\left\{ 
\begin{array}{ll}
\phi,&\text{ solution to }\Delta\phi =0 \text{ in $\Omega$ and }\phi|_{\partial\Omega} = h(T),\\
&\text{ for transparent boundaries,}\\
\vspace{-.5em}\\
0,& \text{ for partially-reflecting boundaries.}
\end{array}
\right.
$$
The respective error operators are defined as $\sfK_a := {\bf Id}_{\Omega_0} - {\bf \Pi}_{\Omega_0}\sfA_a\Lambda_a$ and characterized by the identity $\sfK_a\ff = {\bf\Pi}_{\Omega_0}(w(0),w_t(0))$, for $w$ solution to the backward homogeneous final boundary value problem
\begin{equation}\label{Homan_error}
\left\{\begin{aligned}
(\partial^2_t + a\partial_t- c^2\Delta)w &= 0,&& \text{in }(0,T)\times\Omega,\\
B'(w,0) &= 0 ,&&\text{on }(0,T)\times\partial\Omega,\\
(v,v_t)|_{t=T}&=(u,u_t)|_{t=T}-(\mathcal{P}(h(T)),0),&&\text{in }\Omega.\\
\end{aligned}\right.
\end{equation}
In the context of partially reflecting boundaries, $\sfA_a$ (and consequently $\sfK_a$) follows from directly adapting the time-reversal scheme of \cite{H} to account for Robin conditions. \\

We can similarly adapt the attenuating time-reversal strategy of \cite{P} as follows.
We set ${\AA}_ah:=(v,v_t)|_{t=0}$ for $v$ solution to
\begin{equation}\label{Palacios_TR}
\left\{\begin{aligned}
(\partial^2_t - a\partial_t- c^2\Delta)v &= 0,&& \text{in }(0,T)\times\Omega,\\
B'(v,h) &=0 ,&&\text{on }(0,T)\times\partial\Omega,\\
(v,v_t)|_{t=T}&=(\mathcal{P}(h(T)),0),&&\text{in }\Omega,\\
\end{aligned}\right.
\end{equation}
where the difference with respect to \eqref{Homan_TR} relies on the sign of the damping coefficient. Since this system is solved backward in time, this is, from $t=T$ to $t=0$, the negative sign in front of the damping coefficients makes the problem to be a dissipative one. 
The respective error operator $\KK_a := {\bf Id}_{\Omega_0} - {\bf \Pi}_{\Omega_0}\AA_a\Lambda_a$ is characterized as $\KK_a\ff ={\bf \Pi}_{\Omega_0}(w(0),w_t(0))$, where this time $w$ is a solution to the non-homogeneous final boundary value problem:
\begin{equation}\label{Palacios_error}
\left\{\begin{aligned}
(\partial^2_t - c^2\Delta)w &= -a(u_t+v_t),&&\text{in }(0,T)\times\Omega,\\
B'(w,0) &= 0 ,&&\text{on }(0,T)\times\partial\Omega,\\
(w,w_t)|_{t=T}&=(u,u_t)|_{t=T}-(\mathcal{P}(h(T)),0),&&\text{in }\Omega.\\
\end{aligned}\right.
\end{equation}
The non-homogeneous source plays an important role in the analysis of the reconstruction algorithm via Neumann series.

\subsection{Geometric optics solutions}\label{subsec:GO}
In this section we seek to 
construct a microlocal (up to smoothing error) back-projection operator $\mathcal{A}^{mic}_a$ satisfying that, for a suitable cut-off function $\chi$, $\mathcal{A}^{mic}_a\chi \Lambda_a$ is an elliptic pseudo-differential operator. The ellipticity of the composition will be essential to prove uniqueness and stability of the observation map.

The construction of $\mathcal{A}^{mic}_a$ consists in finding a forward parametrix for the damped wave operator, analyzing the reflection of waves at the dissipative boundary, and finally, microlocally back-propagate the singularities observed from the boundary. We guarantee that all the singularities emanating from $\ff$ reach the observation region (and are subsequently back-propagated) by imposing the visibility condition defined in Section \ref{sec:IP}.

%

Parametrix constructions in photoacoustic tomography are now a standard procedure \cite{SU,SU1,SUsurvey, H,P,SY}, nonetheless, each observation setting brings its own difficulties and details that needs to be carefully addressed. We will focus only on the partially-reflecting boundary setting since the case of transparent boundaries has been studied previously in the literature (see, for instance, \cite{SU,H}).
 

\subsubsection{Forward parametrix (no boundaries).}
Let $\ff=(f_1,f_2)$ be supported in the compact set $\overline{\Omega}_0\subset\Omega$. 
Microlocalizing it, we can assume without loss of generality that its wavefront set is contained in a conic neighborhood of some $(x_0,\xi_0)\in T^*\overline{\Omega}_0$. We look for a solution $u$ to the damped wave equation $\square_a u=0$ in the form
\begin{equation}\label{ansatz}
u(t,x) = (2\pi)^{-n}\sum_{\sigma=\pm}\int e^{i\phi^\sigma(t,x,\xi)}\big(A^\sigma_1(t,x,\xi)\hat{f}_1(\xi) + |\xi|^{-1}A^\sigma_2(t,x,\xi)\hat{f}_2(\xi)\big)d\xi,
\end{equation}
with $\hat{f}(\xi) = \int f(x)e^{-i\xi\cdot x}dx$ the Fourier transform of the function $f(x)$. 

We apply the damped wave operator to \eqref{ansatz} which yields
$$\square_a u = (2\pi)^{-n}\sum_{\sigma=\pm}\int e^{i\phi^\sigma}\big([I^\sigma_{1,0} + I^\sigma_{1,1} + I^\sigma_{1,2}]\hat{f}_1 + |\xi|^{-1}[I^\sigma_{2,0} + I^\sigma_{2,1} + I^\sigma_{2,2}]\hat{f}_2\big)d\xi,$$
where for $j=1,2$ and $\sigma=\pm$,
\begin{align*}
I^\sigma_{j,2} &= - A_j^\sigma((\partial_t\phi^\sigma)^2 - c^2|\nabla_y\phi^\sigma|^2);\\
I^\sigma_{j,1} &= 2i[(\partial_t\phi^\sigma)(\partial_tA^\sigma_j) - c^2\nabla_y\phi^\sigma\cdot\nabla_yA^\sigma_j] + iA^\sigma_j\square_a\phi^\sigma;\\
I^\sigma_{j,0} &= \square_aA^\sigma_j.
\end{align*}
We assume the phase $\phi_\sigma$ is homogeneous of order 1 in $\xi$, while the amplitude functions are assumed to be classical in the sense they are given by asymptotic expansions $A^\sigma_j(t,x,\xi)\sim \sum_{k\geq 0}A^\sigma_{j,k}(t,x,\xi)$, for  $j=1,2,\;\sigma=\pm$, and with $A^\sigma_{j,k}$ smooth functions homogeneous of degree $-k$ in $\xi$. We would like to choose the amplitudes $A^\sigma_{j,k}$ and phase functions $\phi^\sigma$ so that $I^\sigma_{j,0}+I^\sigma_{j,1}+I^\sigma_{j,2}$ vanishes for $j=1,2$ and $\sigma=\pm$.

In order to make $I^\sigma_{j,2}=0$ we solve the eikonal equations 
\begin{align}\label{eikonal}
\left\{\begin{matrix} \mp\partial_t\phi^\pm &=& c|\nabla_x\phi^\pm|\\ \phi^\pm|_{t=0} &=& x\cdot\xi,\end{matrix}\right.
\end{align}
where the initial conditions are chosen so that we recover $\ff$ when taking $t=0$ in \eqref{ansatz}. On the other hand, we obtain $I^\sigma_{j,1}+I^\sigma_{j,0}=0$ by solving a recursive system of equations. For this, we define the vector field
\begin{equation}\label{transport_op}
X^\sigma := 2(\partial_t\phi^\sigma)\partial_t - 2c^2\nabla_x\phi^\sigma\cdot\nabla_x.
\end{equation}
The coefficients of the amplitude functions $A^\sigma_{j}$ must then satisfy 
\begin{equation}\label{transport}
X^\sigma A^\sigma_{j,0} + A^\sigma_{j,0}\square_a\phi^\sigma = 0,\quad\text{and}\quad X^\sigma A^\sigma_{j,k} + A^\sigma_{j,k}\square_a\phi^\sigma =   i \square_aA^\sigma_{j,k-1},
\end{equation}
for all $k\geq 1$. The initial condition for this system of equations are obtained by imposing $(u,u_t)|_{t=0}=\ff$. More explicitly, we need 
$$f_1(x) = (2\pi)^{-n}\int e^{ix\cdot\eta}\big((A^+_1 + A^-_1)\big|_{t=0}\hat{f}_1(\eta) + |\eta|^{-1}(A^+_2 + A^-_2)\big|_{t=0}\hat{f}_2(\eta) \big)d\eta,$$
from which we extract the initial conditions
\begin{equation}\label{A_eq1}
A^+_1 + A^-_1 = 1\quad\text{and}\quad A^+_2 + A^-_2=0,\quad \text{at}\quad t=0;
\end{equation}
and in an analogous fashion,
\begin{align*}
f_2(x) &= (2\pi)^{-n}\int e^{ix\cdot\eta}\big([ic|\eta|(-A^+_1 + A^-_1) + \partial_t(A^+_1 + A^-_1)]\big|_{t=0}\hat{f}_1(\eta)\\
&\hspace{7em}+[ic(-A^+_2 + A^-_2) +|\eta|^{-1}\partial_t(A^+_2 + A^-_2)]\big|_{t=0}\hat{f}_2(\eta) \big)d\eta
\end{align*}
yields the equalities (at $t=0$)
\begin{equation}\label{A_eq2}
\begin{aligned}
ic|\eta|(-A^+_1 + A^-_1) + \partial_t(A^+_1 + A^-_1) &= 0,\\
ic(-A^+_2 + A^-_2) +|\eta|^{-1}\partial_t(A^+_2 + A^-_2) &= 1.
\end{aligned}
\end{equation}
The system of linear equation \eqref{A_eq1}-\eqref{A_eq2} is solved iteratively as follows: for $t=0$ we set
\begin{equation}\label{IC_amplitud1}
\begin{array}{ll} A^+_{1,0} + A^-_{1,0} = 1,& A^+_{1,0} - A^-_{1,0} = 0\\
 A^+_{1,k} + A^-_{1,k}=0,& A^+_{1,k} - A^-_{1,k} = ic^{-1}|\eta|^{-1}\partial_t(A^+_{1,k-1} + A^-_{1,k-1}),\quad k\geq1,  \end{array}
\end{equation}
\begin{equation}\label{IC_amplitud2}
\begin{array}{ll} 
A^+_{2,0} + A^-_{2,0} = 0,& A^+_{2,0} - A^-_{2,0} = i/c\\
A^+_{2,k} + A^-_{2,k} =0,&A^+_{2,k} - A^-_{2,k} = ic^{-1}|\eta|^{-1}\partial_t(A^+_{2,k-1} + A^-_{2,k-1}),\quad k\geq1.
\end{array}
\end{equation}
The transport equations \eqref{transport}, with initial conditions \eqref{IC_amplitud1}-\eqref{IC_amplitud2}, can be solved on integral curves of $X^\sigma$ as longs as the eikonal equation \eqref{eikonal} is solvable. This is only possible, in general, for a small interval of time. To continue with this construction in some positive and small interval $(t_1,t_2)$ one has to solve the same equations for $\phi^\sigma$ and $A^\sigma_{j}$ but with initial conditions this time at $t=t_1>0$, where those conditions come from the previous step. A global solution in $[0,T]$ is then constructed iterating this process a finite number of times. 

Notice also that the approximate solution we constructed is indeed accurate up to a smooth error, this is, in this case $u$ and the real solution differ from each other by a smoothing operator acting on $\ff$, in particular, compact.

If we assume the boundary of $\Omega$ is acoustically transparent we can compute $u$ across $\partial\Omega$. We define the trace operator $F:\ff \to u|_{\RR\times\partial\Omega}$, which is given by the restriction of \eqref{ansatz} to the boundary 
and can be written as $F\ff=F^+\ff+F^-\ff$ with $F^\pm$ Fourier Integral Operators (FIO's) with canonical relations given by the graph of the respective diffeomorphisms
\cite{SUsurvey,SY} 
$$
C_\pm\;:\;(x,\xi)\;\longmapsto\; \big(\pm\tau_{\pm}(x,\xi/|\xi|),\;\gamma_{x,\xi}(\tau_{\pm}(x,\xi/|\xi|))\;,\;\mp|\xi|\;,\;\dot{\gamma}'_{x,\xi}(\tau_{\pm}(x,\xi/|\xi|))\big).
$$
In the case of $c\not\equiv 1$ on $\partial\Omega$  the covector norm $|\cdot|$ should be replaced by the one associated to the metric $c^{-2}(x)dx^2$. 
In addition, $\tau$ is the exit time defined as $\tau_{\pm}:=\inf\{\pm t\geq 0:\gamma_{x,\xi}(t)\in\partial\Omega\}$, and the prime in $\dot{\gamma}'_{x,\xi}$ stands for the projection onto $T^*\partial\Omega$. $F^+$ is elliptic in conic neighborhoods of those $(x,\xi)\in T^*\bar{\Omega}_0$ whose geodesics reach the boundary, thus $\tau_+$ is finite; and analogously for $F^-$ and $\tau_-$. 

By writing instead $F\ff = F_{1} f_1 + F_2 f_2$, it turns out that $F_1$ and $F_2$ are  FIO's of order 0 and $-1$ respectively, both with canonical relations of graph type as above and elliptic in $\Omega_0$ under the visibility condition \ref{def:vis_cond1}. In particular we have \cite{HoIV}
\begin{equation}\label{FIO_reg}
\|F\ff\|_{H^{s}(\RR\times\partial\Omega)}\leq C\|\ff\|_{H^{s}(\Omega_0)\times H^{s-1}(\Omega_0)},\quad s\in\RR.
\end{equation}

\subsubsection{Parametrix at the boundary: reflection of waves.}\label{subsec:bdry_para}
So far, we have obtained a solution $u$ (up to a smooth error) for the damped wave equation in free-space. In order to impose the boundary conditions we need to modify the construction accordingly near the boundary, hence, we switch our attention to a different parametrix construction. In this case we look for a solution to \eqref{PAT}-\eqref{PAT_bc} of the form $u^+=u^+_{in} + u^+_{\text{\em ref}}$. The superscript $+$ here means that we are only considering solutions with wavefront set lying in a neighborhood of a solution $\tau + c(x)|\xi|=0$ to the characteristic equation. In other words, we isolate the portion of the wave field generated by $\ff$ which is associated to the boundary trace $F^+\ff$. The subscripts {\em in} and {\em ref} are used to differentiate respectively between the incoming (to the boundary) part of the wave field and the reflected one. 
 The computations below carry out in the same way for $u^-$ (thus, associated to $\tau - c(x)|\xi|=0$ and boundary trace $F^-\ff$). In order to alleviate the notation we will omit the superscript $+$.

We are only interested in a microlocal representation of the solution to \eqref{PAT}-\eqref{PAT_bc}, then, we only construct a parametrix of $u$ near the the boundary  and with an error given by a compact operator. An equality modulus a compact error will be denoted by $\cong$.

Let's consider boundary normal coordinates $x=(x',x^n)$ in a neighborhood of $x_1\in\partial\Omega$, where the interior of $\Omega$ and its boundary are respectively characterized by $x^n<0$ and $x^n=0$. We look for $u_{in}$ and $u_{\text{\em ref}}$ of the form
\begin{equation}\label{bdary_parametrix}
u_\sigma = (2\pi)^{-n}\int e^{i\varphi_\sigma(t,x,\tau,\eta)}b_\sigma(t,x,\tau,\eta)\hat{h}(\tau,\eta)d\tau d\eta,\quad\sigma = in,\;\text{\em ref},
\end{equation}
where $\hat{h}(\tau,\eta) = \int_{\RR\times\RR^{n-1}}e^{-i(t\tau + x'\cdot\eta)}h(t,x')dtdx'$ is the Fourier transform of a compactly supported distribution $h$ on $\RR\times\RR^{n-1}$. We further assume the wavefront set $WF(h)$ is contained in a small conic neighborhood of some $(t_1,x_1,\tau^1,\eta^1)\in T^*(\RR\times\partial\Omega)$ lying in the {\em conic hyperbolic region}: $c(x)|\eta|<-\tau$. Recall that by assumption $\tau^1=-c(x_1)|\xi^1|<0$, for a covector $\xi^1$ such that $(\xi^1)' = \eta^1$, and strict convexity guarantees that $|\eta^1|<|\xi^1|$.

By applying the wave operator $\square_a$ to the previous ansatz one easily verifies that the phase functions $\varphi_\sigma$, and the amplitudes $b_\sigma$, must satisfy respective eikonal and transport equations similar to \eqref{eikonal} and \eqref{transport}. The phase $\varphi_\sigma$ is assumed to be homogeneous of order 1 in $(\tau,\eta)$ and $b_\sigma$ a classical amplitude of order zero, this is, $b_\sigma\sim\sum_{k\geq 0}b^{(k)}_\sigma$ for smooth functions $b^{(k)}_\sigma$, homogeneous of degree $-k$ in $\tau$ and $\eta$. 

For the incoming wave field $u_{in}$ we take
\[
\varphi_{in}= t\tau+x'\cdot\eta,\quad b_{in}^{(0)} = 1
\quad \text{at}\quad x^n=0,
\]
thus from this choice we get $u_{in}|_{\RR\times\partial\Omega} \cong h$ near $(t_1,x_1)$ (the error in this case corresponds to a compact operator acting on $h$). Regarding the phase function $\varphi_{\text{\em ref}}$ we set
$$
\varphi_{\text{\em ref}} = t\tau+x'\cdot\eta,\quad \text{at}\quad x^n=0,
$$
while the boundary condition for the amplitude of $u_{\text{\em ref}}$ is deduce after imposing $\partial_\nu u + \lambda \partial_tu = 0$ at $\partial\Omega$. The previous phase functions differ from each other in the sign of the normal derivative, being positive for $\varphi_{in}$ and negative for $\varphi_{ref}$. Indeed, since in boundary normal coordinates the outward normal derivative takes the form $\partial_\nu = \frac{\partial}{\partial x^n}$, we have 
\[
\partial_{\nu}\varphi_{in}  = -\partial_{\nu}\varphi_{ref} = \sqrt{c^{-2}(x)\tau^2 - |\eta|^2} \quad\text{at}\quad x^n=0.
\]
The previous ansatz and the phase function $\varphi_{in}$ are used to define the incoming Dirichlet-to-Neumann map $N_{in}$ as the zero-th order $\Psi$DO
\[
N_{in}: u_{in}|_{\RR\times\partial\Omega}\mapsto \partial_\nu u_{in}|_{\RR\times\partial\Omega},
\]
with principal symbol $i\partial_{x^n}\varphi_{in} = i\sqrt{c^{-2}(x)\tau^2 - |\eta|^2}$; the outgoing Dirichlet-to-Neumann map $N_{out}$ is defined similarly in terms of $u_{ref}$ and with a principal symbol given by $i\partial_{x^n}\varphi_{ref} = -i\sqrt{c^{-2}(x)\tau^2 - |\eta|^2}$ (see \cite{SU1, SY} for more details). 

The amplitude function for the reflected wave field is chosen such that at $x^n=0$,
$$
\begin{aligned}
&i (b_{in}\partial_{x^n}\varphi_{in} + b_{\text{\em ref}}\partial_{x^n}\varphi_{\text{\em ref}}) + (\partial_{x^n}b_{in} + \partial_{x^n}b_{\text{\em ref}}) \\
&\hspace{2em}+ \lambda \big( i (b_{in}\partial_{t}\varphi_{in} + b_{\text{\em ref}}\partial_{t}\varphi_{\text{\em ref}}) + (\partial_{t}b_{in} + \partial_{t}b_{\text{\em ref}})\big) = 0.
\end{aligned}
$$
We are only interested in the previous equality at the highest level of homogeneity in $\tau$ and $\eta$, this is, we just require
\[
(b_{in}^{(0)}\partial_{x^n}\varphi_{in} + b_{\text{\em ref}}^{(0)}\partial_{x^n}\varphi_{\text{\em ref}})+ \lambda (b_{in}^{(0)}\partial_{t}\varphi_{in} + b_{\text{\em ref}}^{(0)}\partial_{t}\varphi_{\text{\em ref}})=0,
\]
but recalling the boundary values of $\varphi_{in},\varphi_{\text{\em ref}}$ and $b_{in}$, 
this is satisfied by choosing
\[
b_{\text{\em ref}}^{(0)} = \frac{\partial_{x^n}\varphi_{in} + \tau\lambda}{-\partial_{x^n}\varphi_{\text{\em ref}} - \tau\lambda}
\quad\text{at}\quad x^n=0.
\]
where one verifies that the denominator is non-vanishing since both $-\partial_{x^n}\varphi_{\text{\em ref}}$ and $-\tau$ are positive (the latter because we are following null-bicharacteristics with $\tau=-c(x)|\xi|$).

%
Taking $h= F\ff$ we conclude that the observation operator $\Lambda_a$ is microlocally approximated near the first reflection point by 
$$
\ff\mapsto u_{in}|_{\RR\times\partial\Omega}+u_{\text{\em ref}}|_{\RR\times\partial\Omega},
$$
with $u_{in}$ and $u_{\text{\em ref}}$ as above. We define the reflection operator
\[
R:h\mapsto  u_{\text{\em ref}}|_{\RR\times\partial\Omega},
\]
which is $\Psi$DO of order $0$ with principal symbol $r=b^{(0)}_{ref}(t,x',\tau,\xi')$, and we subsequently define the Dirichlet trace operator as $P:=Id + R$. The latter is a $\Psi$DO on $\partial\Omega$ with a principal symbol 
\[
p(t,x,\tau,\xi):=1+r(t,x,\tau,\xi) = 1 +  \frac{ \sqrt{c^{-2}(x)\tau^2 - |\eta|^2} + \tau\lambda}{ \sqrt{c^{-2}(x)\tau^2 - |\eta|^2} - \tau\lambda},
\]
thus positive (and therefore $P$ elliptic) for $0<c|\eta|<-\tau$. We remark that if $\lambda\not\equiv 0$, the previous operators are elliptic only for $\tau<0$, which relates to the fact that Robin boundary conditions are well-posed only forward in time. We also notice that when $\lambda = 0$ (this is, in the complement of the observation region $\Gamma$) the symbol $r=1$ ---this fact will be important in section \ref{sec:par_bp}.

We then characterize the (Dirichlet) boundary trace of the parametrix at the first reflection point as
\[
\ff\longmapsto PF\ff.
\]
This means $F$ propagates the information to the boundary (and a bit beyond), while $P$ is applied to take into account the partially-reflecting boundary and thus determines the trace of $u$ at the boundary.

Using a similar construction as in \eqref{bdary_parametrix} we define another FIO, $Gh = u_{ref}$, with $h$ microlocalized near a single $(t_1,x_1,\tau_1,\xi'_1)\in T^*(\RR\times\partial\Omega)$. We then solve \eqref{PAT} until the singularities hit $\partial\Omega$ again and slightly beyond, and restrict such solution to $\RR\times\partial\Omega$, thus, its wavefront set is contained in a neighborhood of some $(t_2,x_2,\tau_2,\xi_2')$ with $(t_2,x_2,\tau_2,\xi_2)$ belonging to the same null bicharacteristic as $(t_1,x_1,\tau_1,\xi_1)$, with $\xi_1$ a unit covector pointing inside $\Omega$ and whose projection to $T_{x_1}^*\partial\Omega$ is $\xi_1'$. It is also a zero order FIO with canonical relation corresponding to the graph of the diffeomorphism
\[
C_b : (t,x,\tau,\xi')\mapsto  \left(t+\tau_+(x,\xi/|\xi|), \gamma_{x,\xi}\left(\tau_+(x,\xi)\right), -|\xi|, \dot{\gamma}'_{x,\xi}\left(\tau_+(x,\xi)\right) \right),
\]
with $\xi = (\xi', -\sqrt{c(x)^2\tau^2 - |\xi'|^2})$ in boundary normal coordinates (see also \cite{SY}). 

So far we have only constructed a parametrix at the first reflection point. If several reflections of a singularity issued by $\ff$ occur in $(0,T)\times\partial\Omega$ then $u$ takes a more convoluted form that we explain next. 

Let $F$ and $G$ be the FIO's defined previously. For a singularity issued from $(x_0,\xi_0)$ at the initial time, let's assume that its associated broken geodesic is reflected $m$ times in the time interval $(0,T)$. For a microlocalized $\ff$ as above the boundary trace of the parametrix $u$ (solution to \eqref{PAT}) takes the form
\begin{equation}\label{u+}
\begin{aligned}
u|_{(0,T)\times\partial\Omega} = &\sum_{k=1}^mP\left(GR\right)^{k-1}F\ff.
\end{aligned}
\end{equation}
A simpler way of visualizing this is with the following diagram that show the sequence of boundary traces of $u_+$ near each reflection point:
\begin{equation}\label{bdry_traces}
\ff \longmapsto PF\ff\longmapsto PGRF\ff \longmapsto P(GR)^2F\ff \longmapsto ... \longmapsto P(GR)^{m-1}F\ff.
\end{equation}
We then have $\Lambda_a=\Lambda_a^++\Lambda_a^-$, with 
$\Lambda_a^\pm\ff \cong u^\pm|_{\RR\times\Gamma}$
(i.e., equal modulo a compact error) with $u^+|_{(0,T)\times\Gamma}$ and $u^-|_{(0,T)\times\Gamma}$ as in \eqref{u+}, which means that $\Lambda_a$ is an FIO of order $(0,-1)$ with canonical relation of graph type, and as a consequence (again from  \cite{HoIV}),
\begin{equation}\label{FIO_reg2}
\|\Lambda_a\ff\|_{H^{s}(\RR\times\Gamma)}\leq C\|\ff\|_{H^{s}(\Omega_0)\times H^{s-1}(\Omega_0)},\quad s\in\RR.
\end{equation}

\subsubsection{Microlocal back-projection of singularities.} \label{sec:par_bp}
As mentioned previously, in order to obtain the microlocal stability of the  measurement operator $\Lambda_a$ we need to be able to back-propagate all the singularities that reach our observation region. For this, we intend to back-propagate the boundary data by approximately solving the system
\[
\left\{\begin{aligned}
(\partial^2_t + a\partial_t- c^2\Delta)v &= 0,&& \text{in }(0,T)\times\Omega,\\
\partial_\nu v - \lambda\partial_t v&= -\lambda\partial_t h ,&&\text{on }(0,T)\times\partial\Omega,\\
(v,v_t)|_{t=T}&= (0,0),&&\text{in }\Omega\\
\end{aligned}\right.
\]
(a similar approach appears in \cite{NK}). We achieve this by constructing a back-projection FIO in terms of $G^{-1}$ and $F^{-1}$ ---remember that $F$ and $G$ are elliptic in neighborhoods of covectors lying in the conic hyperbolic region. 
Near the boundary the reflection of singularities is directed by the  Robin-type conditions we impose there, namely, $\partial_\nu v -\lambda \partial_tv=-\lambda \partial_th$, where we will eventually take $h=\chi \Lambda_a\ff$. 
This means that without loss of generality we assume $h$ decays to zero near $t=T$ since this can be obtained via multiplication with the smooth cut-off $\chi$.


At the principal level, the boundary condition takes the form
\[
N_{in}v_{in} + N_{out}v_{ref} - \lambda\partial_t(v_{in}+v_{ref})\cong -\lambda \partial_t h
\]
for a similar decomposition $v=v_{in}+v_{ref}$ and with $N_{in}$ and $N_{out}$ as defined above.
When back-propagating a singularity and due to the condition at $t=T$, the component $v_{ref}$ of the time-reversed wave field carries no singularity, therefore the boundary condition simplifies to $N_{in}v_{in} - \lambda\partial_tv_{in}\cong -\lambda \partial_t h$. An equivalent way of writing this is $v_{in}\cong-(N_{in}-\lambda\partial_t)^{-1}\lambda \partial_t h$, where we notice the operator $N_{in}-\lambda\partial_t$ has a principal symbol 
\[
i\left(\sqrt{c^{-2}(x)\tau^2-|\xi'|^2}-\tau\lambda\right),
\]
which is positive near the null-bicharacteristic (i.e. those satisfying $\tau=-c(x)|\xi|$ which by strict convexity intersect the boundary transversally). We can then invert it up to a $\Psi$DO of negative order (thus smoothing).
To alleviate the notation we write $Q= -(N_{in}-\lambda\partial_t)^{-1}\lambda \partial_t$, which is a zero-th order $\Psi$DO at the boundary and elliptic in the conic hyperbolic region. Its principal symbols is positive at those singularities reaching the observation region $\Gamma=\{\lambda>0\}$ and given by
\[
q(x,t;\xi',\tau) = \frac{-\tau\lambda(x)}{\sqrt{c^{-2}(x)\tau^2-|\xi'|^2}-\tau\lambda}.
\]

Let's assume that $h$ has a wavefront set contained in a neighborhood of some $(t_m,x_m,\tau_m,\xi_m')\in T^*(0,T)\times\Gamma$, with the subindex $m$ representing the fact that $(t_m,x_m)$ is the $m$-th time the broken bicharacteristic passing through $(0,x_0,\tau_0,\xi_0)$ reaches the boundary and is reflected back to the interior of $\Omega$.

We back-propagate the singularities near $(t_m,x_m,\tau_m,\xi_m')$ by applying $F^{-1}$ if $m=1$, otherwise, the
bicharacteristic passing through $(t_m,x_m,\tau_m,\xi_m)$ is propagated back to smaller times until it intersect the boundary at $(t_{m-1},x_{m-1})$. 
The back projection of the Robin data at the first reflection point gives
\[
-\lambda\partial_th|_{\RR\times\Gamma \text{ near 1st reflection pt.}} \longmapsto F^{-1}Qh.
\]
If instead $m>1$, we back-propagate this singularity using $G^{-1}$ until it reaches the boundary. Since we are assuming the wavefront set of $h$ is contained in a conic neighborhood of $(t_m,x_m,\tau_m,\xi'_m)$, then $h$ is smooth near $(t_{m-1},x_{m-1})$, hence the normal derivative of the back-projection must satisfies $\partial_\nu v - \lambda\partial_t v\cong 0$ there.
The wave that is reflected (when going backward in time) has a leading amplitude that is proportional to the incoming wave field (the sign will depend on the magnitude of $\lambda$ on that point). This follows by noticing that by splitting $v$ into an incoming and reflected wave, this is $v=v_{in}+v_{ref}$, they must satisfy at the boundary $(N_{in}-\lambda \partial_t)v_{in}+(N_{out}-\lambda \partial_t)v_{ref} \cong0$, which after applying $(N_{in}-\lambda \partial_t)^{-1}$ to both sides it implies $v_{in}\cong - (N_{in}-\lambda \partial_t)^{-1}(N_{out}-\lambda \partial_t)v_{ref}$. The principal symbol of $- (N_{out}-\lambda \partial_t)^{-1}(N_{in}-\lambda \partial_t)$ is precisely the one of the operator $R$ defined above which means $R\cong - (N_{in}-\lambda \partial_t)^{-1}(N_{out}-\lambda \partial_t)$.

If $m=1$, the backward Robin trace at the second and first reflection points, and the back-projection of this data to $t=0$, are given by
\[
-\lambda \partial_th|_{\RR\times\Gamma \text{ near 2nd reflection pt.}}\longmapsto 0|_{\RR\times\partial\Omega \text{ near 1st reflection pt.}} \longmapsto -F^{-1}(G^{-1}R)Qh.
\]
On the other hand, 
if more reflection occur during the time interval $(0,T)$ we add more reflections points to the previous diagram all with null Robin data at the boundary. In general we obtain
\[
\begin{aligned}
-\lambda \partial_th|_{\RR\times\Gamma \text{ near $m$-th reflection pt.}}&\longmapsto 0|_{\RR\times\partial\Omega \text{ near $(m-1)$-th reflection pt.}}\longmapsto \cdots\\
&\hspace{-5em}\cdots \longmapsto 0|_{\RR\times\partial\Omega \text{ near 1st reflection pt.}}\longmapsto F^{-1}(-G^{-1}R)^mQh.
\end{aligned}
\]
The previous can be generalized to arbitrary distributions $h$ supported in $(0,T)\times\Gamma$ via a partition of unity. The previous defines an FIO of order $(0,1)$ in $(0,T)\times\Gamma$ with canonical relation of graph type that we denote by $\mathcal{A}^{mic}_a$, and which satisfies that for any $s\in \RR$ it is a continuous map
\begin{equation}\label{cont_Aa}
\mathcal{A}^{mic}_a:H^s_{comp}((0,T)\times\Gamma)\to H^s_{loc}(\Omega)\times H^{s-1}_{loc}(\Omega).
\end{equation}

Let's consider $h=u^+|_{\RR\times\Gamma}\cong \Lambda_a^+\ff$ with $\ff$ microlocalized, thus its wave-front set is contained in neighborhoods of the multiple reflection points $\{(t_{n_j},x_{n_j},\tau_{n_j},\xi_{n_j}')\}_{j=1}^m\in T^*(0,T)\times\Gamma$ of the broken geodesic $\gamma_{(x_0,\xi_0)}$. Notice there might be more reflection points associated to the same bicharacteristic (which we denote by $(t_{i},x_{i},\tau_{i},\xi_{i}')$), however, we only consider those that reach the observation set $\Gamma$. With the aid of a microlocal partition of the unity, we write $h = \sum_{j=1}^mh_j$ with each $h_j$ having wavefront set in a neighborhood of its respective $(t_{n_j},x_{n_j},\tau_{n_j},\xi_{n_j}')$. We then back-propagate each of the $h_j$ independently (as done above) and obtain
\begin{equation}\label{parametrix_Aa}
\begin{aligned}
\mathcal{A}^{mic}_a\chi\Lambda_a^+\ff \;&\cong\; \sum^{m}_{j=1}F^{-1}(G^{-1}R)^{{n_j}-1}Q\chi h_j\;\\
&\cong\; \sum^{m}_{j=1}F^{-1}(G^{-1}R)^{{n_j-1}}Q\chi P(GR)^{{n_j-1}}F\ff.
\end{aligned}
\end{equation}

In order to analyze the symbol of the resulting operator let's consider the following notation: for a function $a(x,t,\tau,\xi)$ we write $a_j = a(t_{{n_j}},x_{{n_j}},\tau_{{n_j}},\xi'_{{n_j}})$; and we also notice that when the bicharacteristic hits the boundary outside of $\Gamma$ then $r=1$ there, therefore 
\[
\prod^{n_j-1}_{i=1} r(t_{{i}},x_{{i}},\tau_{i},\xi'_{i})= \prod^{j-1}_{i=1}r(t_{{n_i}},x_{{n_i}},\tau_{{n_i}},\xi'_{{n_i}}) =\prod^{j-1}_{i=1}r_i.
\]
The last expression is taken to be 1 for $j=1$.

A multiple application of Egorov's theorem \cite{HoIV} and the previous allow us to deduce that $\mathcal{A}^{mic}_a\chi\Lambda_a^+$ is a order zero $\Psi$DO with principal symbol
\[
\begin{aligned}
(x_0,\xi_0)\; \longmapsto\; & \sum^{m}_{j=1}\chi_{j}q_jp_{j}\prod^{j-1}_{i=1}r_i^2 
\end{aligned}
\]
which is non-null for visible singularities. We conclude that the operator $\mathcal{A}^{mic}_a\chi\Lambda_a^+$ is elliptic in $\Omega_0$ under the visibility condition \ref{def:vis_cond2}. An analogous argument implies the ellipticity of $\mathcal{A}^{mic}_a\chi\Lambda_a^-$ near visible singularities propagating according to the negative sound speed (i.e., those associated to null-bicharacteristic satisfying $\tau-c(x)|\xi|=0$).

\section{Bounded geometry with transparent-boundary}

\subsection{Uniqueness and stability (Proof of Theorem \ref{thm:uniq_stab}: part 1)}\label{subsec:uniq_bdd}
Let $\chi\in C^\infty_0(\RR\times\partial\Omega)$ be such that  $\supp(\chi)\subset [0,T)\times\Gamma$ and $\chi = 1$ in $[0,T_0]\times\Gamma_0$, for some $T_0<T$ so that the visibility condition (Definition \ref{def:vis_cond1}) still holds for $(\Gamma_0,T_0,\Omega_0)$. Let $\sfA_a$ be the back-projection operator defined as $\sfA_ah := (v,v_t)|_{t=0}$ with $v$ solution to
$$
\left\{
\begin{aligned}
v_{tt}-c^{2}(x)\Delta v + a(x)v_t&=0, &&\text{in }(0,T)\times\Omega,\\
v &= h,&&\text{on }(0,T)\times\partial\Omega,\\
(v,v_t)|_{t=T}&= (0,0),&&\text{in }\Omega.\\
\end{aligned}
\right.
$$
The boundary data is given by $h=\chi\Lambda_a\ff = \chi u|_{(0,T)\times\Gamma}$, where $u$ satisfies
\begin{equation}\label{PAT_sb}
\left\{
\begin{aligned}
u_{tt}-c^{2}(x)\Delta u + a(x)u_t&=0, &&\text{in }(0,T)\times\RR^n,\\
(u,u_t)|_{t=0}&= \ff,&&\text{in }\RR^n.\\
\end{aligned}
\right.
\end{equation}
The composition $\sfA_a\chi\Lambda_a$ is known to be a classical $\Psi$DO of order zero and elliptic under the visibility condition \ref{def:vis_cond1} (see \cite[Theorem 4.1]{H}), therefore, we can find a properly supported $\Psi$DO of order zero such that ${\bf Q} \sfA_a\chi\Lambda_a = \text{\bf Id}+{\bf K_0}$ in a neighborhood of the compact set $\overline{\Omega}_0$, and with ${\bf K_0}$ a smoothing operator. 
Applying ${\bf Q} \sfA_a\chi$ to $\Lambda_a\ff$ 
and rearranging terms we obtain
$$
\ff = {\bf Q} \sfA_a\chi \Lambda_a\ff -{\bf K_0}\ff,
$$
thus, taking the $H^{s}(\Omega_0)\times H^{s-1}(\Omega_0)$-norm with $s\in\RR$ leads to
$$
\|\ff\|_{H^{s}(\Omega_0)\times H^{s-1}(\Omega_0)}\leq \|{\bf Q}\sfA_a\chi \Lambda_a \ff\|_{H^{s}(\Omega_0)\times H^{s-1}(\Omega_0)} +\|{\bf K_0}\ff\|_{H^{s}(\Omega_0)\times H^{s-1}(\Omega_0)} .
$$
The map ${\bf K_0}:H^{s-1}(\Omega_0)\times H^{s-2}(\Omega_0)\to H^{s}(\Omega_0)\times H^{s-1}(\Omega_0)$ is continuous. On the other hand, 
${\bf Q}$ is a zero-th order elliptic $\Psi$DO and $\sfA_a$ is an FIO of order $(0,1)$ with canonical relation of graph type, this implies
\begin{equation}\label{microlocal_stability}
\begin{aligned}
\|\ff\|_{H^{s}(\Omega_0)\times H^{s-1}(\Omega_0)}&\leq C\|\chi\Lambda_a\ff\|_{H^s((0,T)\times\partial\Omega)} +C\|\ff\|_{H^{s-1}(\Omega_0)\times H^{s-2}(\Omega_0)}.
\end{aligned}
\end{equation}

We use the stable recovery of singularities depicted by the previous inequality to obtain uniqueness and we do so by following arguments from \cite{BLR}. Once uniqueness has been established the stability inequality is derived from a standard compactness-uniqueness argument.

Let $\mathcal{N}$ denote the subspace of $H^1((0,T)\times\Omega)$ consisting of all 
invisible solutions to \eqref{PAT_sb}, this is, those satisfying $u|_{(0,T)\times\Gamma} = 0$ and with initial state $(u,u_t)|_{t=0}\in \mathcal{H}(\Omega_0)$. 
From the previous inequality we have that for any $u\in\mathcal{N}$,
$$
\|(u,u_t)|_{t=0}\|_{\mathcal{H}(\Omega_0)} \leq C\|(u,u_t)|_{t=0}\|_{H^{0}(\Omega_0)\times H^{-1}(\Omega_0)},
$$
where the inclusion $\mathcal{H}(\Omega_0)\hookrightarrow H^{0}(\Omega_0)\times H^{-1}(\Omega_0)$ is compact, thus, Riesz's Lemma leads us to conclude that $\mathcal{N}$ is a finite-dimensional subspace. 

Given an arbitrary $u\in\mathcal{N}$, since $(u,u_t)|_{t=0}\in \mathcal{H}(\Omega_0)$, and using again \eqref{microlocal_stability} but this time for $s=2$, we get $(u,u_t)|_{t=0}\in H^2(\Omega_0)\times H^1(\Omega_0)$. Consequently, $(u(t),u_t(t))\in H^2(\Omega)\times H^1(\Omega)$ for all $t\in[0,T]$ and in particular $\partial_tu \in H^1((0,T)\times\Omega)$ (this follows, for instance, by the FIO properties of the map $\ff\mapsto u$). Due to the invariance of the damped wave equation in \eqref{PAT_sb} under time-differentiation, $u_t$ is also a solution vanishing on $[0,T]\times\Gamma$, therefore, $\partial_tu\in\mathcal{N}$. This means $\partial_t$ is a linear operator mapping the finite-dimensional space $\mathcal{N}$ onto itself.

Let's first consider the case of partial data and $\Omega_0\Subset\Omega$. Assuming there is an eigenvalue $\kappa\in\CC$ of $\partial_t$ with a non-trivial eigenfunction $u\in\mathcal{N}$, hence $\partial_tu = \kappa u$, the eigenfunction must take the form $u=e^{\kappa t}f$, with $f=u|_{t=0}$ the initial state, which in addition of being compactly supported (inside $\bar{\Omega} _0$), it has to solve the elliptic equation
$$
-c^{2}\Delta f + (\kappa^2 +\kappa a(x))f = 0\quad\text{in}\quad\Omega.
$$
It follows from the unique continuation property of elliptic operators that the only solution to this is $f\equiv 0$, and consequently, there is no (nontrivial) eigenfunction of $\mathcal{N}$ from where one concludes that $\mathcal{N} = \{0\}$.

On the other hand, for $\Omega_0=\Omega$ and $\Gamma=\partial\Omega$, we known that $f|_{\partial\Omega}=0$, but this is not enough to deduce that it vanishes everywhere. However, we can use $u|_{[0,T]\times\partial\Omega}$ to determine the Neumann data $\partial_\nu u|_{[0,T]\times\partial\Omega}$ by solving the exterior initial boundary value problem (see e.g. \cite[\S 6.1]{SUsurvey}) and conclude that $\partial_\nu u|_{[0,T]\times\partial\Omega} = 0$. This subsequently implies $\partial_\nu f|_{\partial\Omega}=0$ and consequently that $f=0$. The last step is a result of uniqueness for the elliptic Cauchy data problem.

Now that we know the injectivity of the observation map holds, we go back to inequality \eqref{microlocal_stability} where we can apply a well-known compactness-uniqueness argument (see, for instance, \cite{CO} for more details) to deduce the stability inequality
$$
\|\ff\|_{\mathcal{H}(\Omega_0)}\leq C\|\Lambda_a\ff\|_{H^1((0,T)\times\Gamma)},
$$
for some other constant $C>0$.

\subsection{Reconstruction}\label{subsec:rec_bdd}
\subsubsection{Complete data: revisiting \cite{P}}
Let's assume that observations take place in the whole boundary of the region $\Omega$, this is, we set $\Lambda_a\ff = u|_{(0,T)\times\partial\Omega}$. In \cite{P}, a Neumann series reconstruction formula was obtained by virtue of the time-reversal operator $\Pi_1\AA_a$ (recall $\Pi_1(f_1,f_2)=f_1$), for $\AA_ah = (v,v_t)|_{t=0}$ with $v$ solution to the dissipative back-projection system
\begin{equation}\label{Diss_TR}
\left\{\begin{aligned}
(\partial^2_t - a\partial_t- c^2\Delta)v &= 0,&&\text{in }(0,T)\times\Omega,\\
v &=  h ,&&\text{on }(0,T)\times\partial\Omega,\\
(v,v_t)|_{t=T}&=(\phi,0),&&\text{in }\Omega,\\
\end{aligned}\right.
\end{equation}
and $\phi$ the harmonic extension of $h(T)$.
The resulting error operator is given by $K_a := \Pi_1\KK_a$, which assigns $\Pi_1\KK_a\ff = w(0)$ for $w$ solution to 
\begin{equation}\label{error_system_sb}
\left\{\begin{aligned}
(\partial^2_t - c^2\Delta)w &= -a(u_t+v_t),&&\text{in }(0,T)\times\Omega,\\
w &= 0 ,&&\text{on }(0,T)\times\partial\Omega,\\
(v,v_t)|_{t=T}&=(u-\phi,u_t)|_{t=T},&&\text{in }\Omega.\\
\end{aligned}\right.
\end{equation}
Introducing the functional space $H_{D,a}(\Omega)$ as the completion of $C^\infty_0(\Omega)$ under the norm
$$
\|f\|_{H_{D,a}} :=\int_\Omega |\nabla f|^2 + c^{-2}|af|^2dx,
$$
it was stated that for a non-trapping manifold $(\Omega,c^{-2}dx^2)$ ---this is, such that the visibility conditions holds throughout the whole domain and boundary--- the operator $K_a$ is a contraction in $H_{D,a}(\Omega)$. 
In the derivation of such result the next inequality was used,
$$\|K_a\ff\|^2_{H_{D,a}} = \|w(0)\|^2_{H_{D,a}}\leq E_\Omega(\ww(0)),$$ 
however, it is not clear that this inequality holds in general since it imposes a precise relation between the norms of $w_t(0)$ and $aw(0)$. The comparison of these two functions is not evident. 

The right way to proceed is by generalizing the analysis to both components of $\ww$ and consider then the full back-projection operator $\AA_a$, and of course the error operator $\KK_a = \text{\bf Id} - \AA_a\Lambda_a$. The previous inequality is superseded by the trivial equality $\|\KK_a\ff\|^2_{\mathcal{H}(\Omega)} = E_\Omega(\ww(0))$. 
Here we assume $\Omega_0 = \Omega$ satisfies the visibility condition, or equivalently that $(\Omega,c^{-2}dx^2)$ is non-trapping (i.e. $T_1(\Omega,\partial\Omega)<\infty$). 

Following the computations carried out in \cite{P}, one easily verifies that $\KK_a$ is a contraction over $\mathcal{H}(\Omega)$, 
provided $\partial\Omega$ is strictly convex and measurements are taken all over the boundary in such a way that every singularity issued from $\Omega$ is visible in finite time from $\partial\Omega$. We then obtain the reconstruction result below. Compared to \cite{P}, there is an improvement in the lower bound for the observation time needed to achieve reconstruction, which now matches the one for the unattenuated case (see \cite{SU1} or \cite{QSUZ}). The proof is essentially the same and the improvement follows after noticing that the microlocal analysis needed to prove Proposition \ref{prop:energy_ineq_sf} simplifies in the case of geodesics associated to singularities of the initial condition with only one branch of it reaching $\Gamma$ (thus, the other one still trapped inside $\Omega$ at time $T$). For more details we refer to  \cite{P} and the proof of the analogous Proposition \ref{prop:energy_ineq} for the partially-reflecting boundary case.

\begin{theorem}\label{thm:reconstruction_freespace}
Let $(\Omega,c^{-2}dx^2)$ be a non-trapping manifold with $\partial\Omega$ smooth and strictly convex, and let $T>\frac{1}{2}T_1(\Omega,\partial\Omega)$. 
The operator ${\bf K}_a$ is a contraction in $\mathcal{H}(\Omega)$ and we get the following reconstruction formula for the photoacoustic problem \eqref{PAT}:
$$\ff = \sum^\infty_{m=0}{\bf K}_a^m\AA_ah,\quad\quad h:=\Lambda_a\ff.$$
\end{theorem}
\begin{proof}
Let's analyze the energy inequality associated to \eqref{error_system_sb}, where we recall our notation $\ww = (w,w_t)$. We multiply by $c^{-2}w_t$ and integrate over $(0,T)\times\Omega$. Then, integration by parts yields
\begin{equation}\label{ineq_w_sb}
\begin{aligned}&E_\Omega(\ww(0)) \\
&= E_\Omega(\ww(T)) + 2\int_{(0,T)\times\Omega} ac^{-2}(u_t+v_t)(u_t - v_t)dtdx\\
&= E_\Omega(\ww(T)) + 2\int_{(0,T)\times\Omega}  ac^{-2}|u_t|^2dtdx - 2\int_{(0,T)\times\Omega} ac^{-2}|v_t|^2dtdx\\
&\leq E_\Omega(\uu(T)) + 2\int_0^T\int_\Omega ac^{-2}|u_t|^2dtdx - \|\phi\|^2_{H_D(\Omega)}\\
&\leq \mathcal{E}_{\Omega,T}(\uu).
\end{aligned}
\end{equation}
The conclusion of the theorem follows directly from the next estimate which was proven in \cite{P}. We also refer the reader to Proposition \ref{prop:energy_ineq} and its proof below, which states the analogous inequality in the case of a partially reflecting boundary. 

\begin{proposition}\label{prop:energy_ineq_sf}
 Let $\uu$ be a solution of \eqref{PAT_sb} with initial condition $\ff\in\mathcal{H}(\Omega)$. There exists $C(T)>1$ so that
$$\|\ff\|^2_{\mathcal{H}(\Omega)} \leq CE_{\RR^n\backslash\Omega}(\uu(T)).$$
\end{proposition}
Recalling that the damping coefficient is supported inside $\Omega$ we see that the energy estimate associated to $\uu$ gives
$$
\|\ff\|^2_{\mathcal{H}(\Omega)}= E_\Omega(\uu(0)) = \mathcal{E}_{\Omega,T}(\uu) + E_{\RR^n\backslash\Omega}(\uu(T)).$$
We then use Proposition \ref{prop:energy_ineq_sf} to estimate $\mathcal{E}_{\Omega,T}(\uu)$ from above, leading to
$$
\mathcal{E}_{\Omega,T}(\uu) = \|\ff\|^2_{\mathcal{H}(\Omega)} -E_{\RR^n\backslash\Omega,T}(\uu)\leq (1-C^{-1})\|\ff\|^2_{\mathcal{H}(\Omega)}. 
$$
Bringing this together with \eqref{ineq_w_sb}, and noticing that $\|\KK_a\ff\|^2_{\mathcal{H}(\Omega)}= E_{\Omega}(\ww(0))$, the previous implies
$$
\|\KK_a\ff\|^2_{\mathcal{H}(\Omega)}\leq (1-C^{-1})\|\ff\|^2_{\mathcal{H}(\Omega)}.
$$
This means $\KK_a$ is a contraction in $\mathcal{H}(\Omega)$ and consequently $\AA_a\Lambda_a = \text{\bf Id} - \KK_a$ is invertible via a Neumann series.
\end{proof}

\subsubsection{Partial data}

It is not hard to see that the previous reconstruction procedure ---based on time-reversal--- doesn't naturally extend to the partial data case (at least in the transparent-boundary geometry). The fact we are not able to observe on the whole boundary may lead us to lose important low frequency information, even when the high frequency component is well captured under the (microlocal) visibility assumption. This portion of the energy that escapes the boundary detection region, and thus not observed, creates  difficulties when deciding what condition to impose on the rest of the boundary (outside the observation part) during the time-reversal process. Indeed, for the time-reversal step, there is no boundary condition that can guarantee that the error system (the one satisfied by $w$) is energy-dissipative or at least energy-preserving.

The reconstruction in this setting was addressed recently in \cite{HaN}. Regardless the fact that at the time the injectivity of the observation map was still an open question, the authors stablished the convergence of an iterative method under the hypothesis of the visibility condition and  injectivity. In consequence, our result in Theorem \ref{thm:uniq_stab} guarantees the validity of the reconstruction algorithm proposed in \cite{HaN} for large enough observation times. Their reconstruction scheme is based on studying the adjoint operator associated to the forward problem \eqref{PAT} (for $U=\RR^n$), which consists in solving a dissipative system (similar to \eqref{Diss_TR}) in the whole space and with a source term supported on $\partial\Omega$, which of course depends on the boundary observations. 

No Neumann series formula has been proven to converge for the partial data case even in the context of unattenuated media. A discussion about the difficulties encountered on this matter can be found in \cite{QSUZ}. Nevertheless, by removing the boundedness condition over $\Omega$ and assuming there is dissipation of energy across the observation set $\Gamma\subset\partial\Omega$ one can indeed prove the existence of a Neumann series formula as we will see next. This is the content of Theorem \ref{thm:reconstruction} in the next section.

\section{Unbounded geometry with partially--reflecting boundary}
Most of the computations in this section resemble the ones presented in the previous case. Uniqueness and stability are obtained by following what we did previously almost step-by-step. The differences rely in the proof of reconstruction, where the presence of the reflecting boundary require to analyze the behavior of  the propagating wave field near the boundary. 
\subsection{Uniqueness and stability (Proof of Theorem \ref{thm:uniq_stab}: part 2)}\label{subsec:uniq_unbdd}
Let $\chi\in C^\infty_0(\RR\times\partial\Omega)$ be such that  $\supp(\chi)\subset[0,T)\times\Gamma$, and $\chi=1$ in $[0,T_0]\times\Gamma_0$ for some $T_0<T$ for which the visibility condition \ref{def:vis_cond2} still holds for $(\Gamma_0,T_0,\Omega_0)$. Let $\mathcal{A}^{mic}_a$ be the microlocal back-projection operator constructed in section \ref{subsec:GO} which back-propagates the boundary data $h=\chi\Lambda_a\ff$ by imposing Robin boundary conditions.

The analysis carried out in \ref{subsec:GO} ---in particular, the ellipticity of $\mathcal{A}^{mic}_a\chi\Lambda_a$--- allows us to deduce the existence of a properly supported $\Psi$DO of order zero, ${\bf Q}$, such that ${\bf Q}\mathcal{A}^{mic}_a\chi\Lambda_a = \text{\bf Id}_{\Omega_0}+\KK_0$ in a neighborhood of the compact set $\overline{\Omega}_0$ and with $\KK_0$ a smoothing operator. Then
$$
\ff = {\bf Q}\mathcal{A}^{mic}_a\chi \Lambda_a\ff -\KK_0\ff,
$$
and therefore,
$$
\|\ff\|_{H^s(\Omega_0)\times H^{s-1}(\Omega_0)}\leq \|{\bf Q}\mathcal{A}^{mic}_a\chi \Lambda_a \ff\|_{H^s(\Omega_0)\times H^{s-1}(\Omega_0)} +\|{\bf K_0}\ff\|_{H^s(\Omega_0)\times H^{s-1}(\Omega_0)} .
$$
We are in a similar situation as in the transparent-boundary case since $\KK_0:H^{s-1}(\Omega_0)\times H^{s-2}(\Omega_0)\to H^s(\Omega_0)\times H^{s-1}(\Omega_0)$ is continuous, ${\bf Q}$ is a zero-th order elliptic $\Psi$DO, and $\mathcal{A}^{mic}_a$ is an FIO of order $(0,1)$ with canonical relation of graph type, thus, satisfying \eqref{cont_Aa}. 
We then obtain
\begin{equation}\label{microlocal_stability2}
\begin{aligned}
\|\ff\|_{H^s(\Omega_0)\times H^{s-1}(\Omega_0)}&\leq C\|\chi\Lambda_a\ff\|_{H^s((0,T)\times\Gamma)} +C\|\ff\|_{H^{s-1}(\Omega_0)\times H^{s-2}(\Omega_0)}.
\end{aligned}
\end{equation}
The same argument used previously ---Riesz's Lemma and unique continuation for elliptic operators--- can be applied here to deduce the injectivity and stability of the inverse problem under the visibility condition.

\subsection{Reconstruction}\label{subsec:rec_unbdd}
Let $u$ be solution to the attenuating system \eqref{PAT}-\eqref{PAT_bc} (for $U=\Omega$, unbounded) and consider $\AA_a$ and $\KK_a$, the time-reversal and error operators defined in Section \ref{sec:TR}. The main result of this section is the next.
\begin{theorem}\label{thm:reconstruction}
Assume $\Gamma$ is a strictly convex surface for the metric $c^{-2}dx^2$ and the visibility condition hold for $T>0$, $\Omega_0\Subset\Omega$, and $\Gamma$. 
Then, ${\bf K}_a$ is a contraction in $\mathcal{H}(\Omega_0)$ and we get the following reconstruction formula for the photoacoustic problem \eqref{PAT}:
$$\ff = \sum^\infty_{m=0}{\bf K}_a^m{\bf\Pi}_{\Omega_0}\AA_ah\quad h:=\Lambda_a\ff.$$
\end{theorem}

\begin{proof}
The error function $\ww$ solves \eqref{Palacios_error} with null Neumann conditions. The same energy computations as in the transparent boundary case lead us to the inequality 
\begin{equation}\label{ineq_w}
E_\Omega(\ww(0)) \leq
\mathcal{E}_{\Omega,T}(\uu).
\end{equation}
The conclusion of the theorem follows directly from the next estimate whose proof we skip for a moment.
\begin{proposition}\label{prop:energy_ineq}
 Let $\uu$ be a solution of \eqref{PAT}-\eqref{PAT_bc} in $\Omega$ with initial condition $\ff\in\mathcal{H}(\Omega_0)$. Assuming the hypothesis of the Theorem \ref{thm:reconstruction}, there exists $C(\Omega_0,T)>1$ so that
$$\|\ff\|^2_{\mathcal{H}(\Omega_0)} \leq C\int_{(0,T)\times\partial\Omega}\lambda|u_t|^2dtdS,$$
with $dS$ the surface measure on $\partial\Omega$.
\end{proposition}
The energy estimate for $\uu$ then gives
$$
\|\ff\|^2_{\mathcal{H}(\Omega_0)}= E_\Omega(\uu(0)) = \mathcal{E}_{\Omega,T}(\uu) + \int_{(0,T)\times\partial\Omega}\lambda|u_t|^2dtdS,$$
and we use Proposition \ref{prop:energy_ineq} to estimate from above $\mathcal{E}_{\Omega,T}(\uu)$. We get
$$
\mathcal{E}_{\Omega,T}(\uu) = \|\ff\|^2_{\mathcal{H}(\Omega_0)} - \int_{(0,T)\times\partial\Omega}\lambda|u_t|^2dtdS\leq (1-C^{-1})\|\ff\|^2_{\mathcal{H}(\Omega_0)}. 
$$
Plugging this into \eqref{ineq_w} and noticing that $\KK_a\ff = \Pi_{\Omega_0}\ww(0)$ thus $\|\KK_a\ff\|^2_{\mathcal{H}(\Omega_0)}\leq E_{\Omega}(\ww(0))$, then
$$
\|\KK_a\ff\|^2_{\mathcal{H}(\Omega_0)}\leq (1-C^{-1})\|\ff\|^2_{\mathcal{H}(\Omega_0)},
$$
and we conclude $\KK_a$ is a contraction in $\mathcal{H}(\Omega_0)$. The inversion of ${\bf\Pi}_{\Omega_0}\AA_a\Lambda_a = \text{\bf Id}_{\Omega_0} - \KK_a$ via a Neumann series follows directly.
\end{proof}

\begin{proof}[Proof of Proposition \ref{prop:energy_ineq}]
Let $\Gamma_0=\{\lambda\geq \lambda_0>0\}\subset\Gamma$ such that the visibility condition still holds for $(\Gamma_0,T,\Omega_0)$. 
We suppose for a moment that $WF(\ff)$ lies on a small conic neighborhood of some $(x_0,\xi_0)\in T^*\overline{\Omega}_0$. Assuming both branches of the geodesic issued from $(x_0,\xi_0)$ reach the observation region $\Gamma_0$ in time less than $T$, we denote by $(t_{n_j}^\pm,x_{n_j}^\pm)$ the times and points where those broken-geodesic segments make contact with $\Gamma_0$. The case of only one part of the geodesic reaching $\Gamma_0$ is simpler and follows from similar computations as the ones we present below. This is because in this case the observation map has only one component, namely, $\Lambda_a^+\ff$ or $\Lambda_a^-\ff$. 

We desire to estimate the energy deposited on $\Gamma$ due to the dissipative (Robin) conditions imposed there for which we use the boundary parametrix construction of Section \ref{subsec:GO}. Up to a compact operator acting on $\ff$ we have
$\Lambda_a\ff
\cong u^+|_{\RR\times\partial\Omega}+ u^-|_{\RR\times\partial\Omega}$, thus, denoting $h^\pm =u^\pm|_{\RR\times\partial\Omega}$ we see that
$$
\begin{aligned}
-\mathfrak{Re}\int_{(0,T)\times\partial \Omega}u_t\overline{\partial_\nu u}dtdS &=\mathfrak{Re}\int_{(0,T)\times\partial \Omega}\lambda |\partial_tu|^2dtdS \\
&\cong \mathfrak{Re}\sum_{\sigma=\pm}\int_{(0,T)\times\partial \Omega}\lambda|\partial_th^\sigma|^2dtdS, \\
\end{aligned}
$$
where $\langle\cdot,\cdot\rangle$ stands for the inner product in $L^2(\RR\times\partial\Omega)$ while the symbol $\cong$ stands for equality up to an error bounded by lower order norms of $\ff$, thus compact. 
Notice that, since
$$WF(h^+)\cap WF(h^-)=\emptyset,$$
any cross terms between $h^+$ and $h^-$ is estimated by lower order norms of $\ff$. Indeed, considering a microlocal cut-off $\tilde{\chi}$ supported around $WF(h^-)$ and with disjoint intersection with $WF(h^+)$ we have that $\tilde{\chi}(D)\partial_th^+, (\text{Id}-\tilde{\chi}(D))\partial_th^-\in C^\infty$, in other words, both can be regarded as smoothing operators applied to the initial condition $\ff$. Then, and recalling also the mapping properties of the FIO's $\Lambda_a^\pm$, for any $s<0$,
\begin{equation}\label{error_compact}
\begin{aligned}
|\langle \lambda\partial_t h^+, \partial_th^-\rangle| &\leq |\langle \lambda(\text{Id}-\tilde{\chi}(D))\partial_th^+, \partial_th^-\rangle| +  |\langle \lambda\chi(D)\partial_th^+, \partial_th^-\rangle|  \\
&\leq C\| (\text{Id}-\tilde{\chi}(D))\partial_th^-\|_{H^{-s}}\|\partial_t h^+\|_{H^{s}} \\
&\hspace{3em}+ C\| \tilde{\chi}(D)\partial_th^+\|_{H^{-s}}\|\partial_th^-\|_{H^{s}}\\
&\leq C\|\ff\|^2_{H^{s}(\Omega_0)\times H^{s-1}(\Omega_0)},\\
\end{aligned}
\end{equation}
with a final constant depending on $s$. 
In consequence, we deduce
\[
-\mathfrak{Re}\int_{(0,T)\times\partial \Omega}u_t\overline{\partial_\nu u}dtdS\geq \frac{1}{C}\big(\|\lambda^{1/2}\partial_th^+\|^2_{L^2} + \|\lambda^{1/2}\partial_th^-\|^2_{L^2}\big) - C\|\ff\|_{H^0(\Omega_0)\times H^{-1}(\Omega_0)}.
\]
To continue, we notice that on $\Gamma_0$, $\partial_th^+=\partial_tu^+|_{\RR\times\Gamma_0}\cong -\lambda^{-1}\partial_\nu u^+|_{\RR\times\Gamma_0}$, thus, on each neighborhood of the space-time point where the geodesic issued from $(x_0,\xi_0)$ hits the boundary and is reflected we have that
\[
\partial_th^+|_{\text{near the $m$-th reflection point on $(0,T)\times\Gamma_0$}}\cong -\lambda^{-1}N_{in}u^+_{in} - \lambda^{-1}N_{out}u^+_{ref},
\]
with $N_{in}$ and $N_{out}$ the respective incoming and outgoing microlocal Dirichlet-to-Neumann maps. Below we use that $N_{in}\cong -N_{out}$.

Recalling the reflection operators $R^\pm$ and the FIO's $G^\pm$ and $F^\pm$ we have that near a singularity
\[
N_{in}u^+_{in} + N_{out}u^+_{ref} \cong N_{in}(u^+_{in} - u^+_{ref})= N_{in}(I-R^+)g,
\]
for $g=\left(G^+R^+\right)^{m-1}F^+\ff$ and some $m\geq 1$, and noticing that $h^+\cong P^+g = (I+R^+)g$ ($P^+$ is the trace operator defined in section \ref{subsec:bdry_para}), then
\[
N_{in}u^+_{in} + N_{out}u^+_{ref} \cong N_{in}(u^+_{in} - u^+_{ref})= N_{in}(I-R^+)(I+R^+)^{-1}h^+.
\]
By denoting $\mathcal{N}^+=N_{in}(I-R^+)(I+R^+)^{-1}$ we have
\[
\|\lambda^{1/2}\partial_th^+\|^2_{L^2}\cong\langle \lambda^{-1}(\mathcal{N}^+)^*\mathcal{N}^+h^+,h^+\rangle,
\]
where $\mathcal{N}^+$ is a $\Psi$DO of order 1, and elliptic on $c^{-2}(x)\tau^2>|\xi'|^2$ since it is a composition of three elliptic operators (we use here that the principal symbol of $R^\pm$ take values in $(-1,1)$). We similarly define $\mathcal{N}^-$.
The previous implies
\[
\|\lambda^{1/2}\partial_th^+\|^2_{L^2} + \|\lambda^{1/2}\partial_th^-\|^2_{L^2} \cong \langle \lambda^{-1}(\mathcal{N}^+)^*\mathcal{N}^+h^+,h^+\rangle + \langle \lambda^{-1}(\mathcal{N}^-)^*\mathcal{N}^-h^-,h^-\rangle,
\]
with $\lambda^{-1}(\mathcal{N}^\pm)^*\mathcal{N}^\pm$ an elliptic $\Psi$DO's in the hyperbolic region $c^{2}(x)\tau^2>|\xi'|^2$. 
Their principal symbols are positive and bounded from below by $|(\tau,\xi)|^2$, 
thus, by Garding's inequality we obtain
\[
\begin{aligned}
\|\lambda^{1/2}\partial_th^+\|^2_{L^2} + \|\lambda^{1/2}\partial_th^-\|^2_{L^2}&\geq\frac{1}{C}\left(\|h^+\|^2_{H^1}+\|h^-\|^2_{H^1}\right) - C\left(\|h^+\|^2_{L^2}+ \|h^-\|^2_{L^2}\right)\\
&\geq \frac{1}{2C}\|h\|^2_{H^1} - C\left(\|h^+\|^2_{L^2}+ \|h^-\|^2_{L^2}\right)
\end{aligned}
\]
for some $C>0$ depending on $\Omega_0$. 

We then bring together all the previous, along with the continuity properties of $\Lambda_a^{\pm}$ and the stability inequality of Theorem \ref{thm:uniq_stab}, to deduce 
\begin{equation}\label{ineq_f1}
\|\ff\|^2_{H^1(\Omega_0)\times H^0(\Omega_0)}\leq C\left(-\mathfrak{Re}\int_{(0,T)\times\partial \Omega}u_t\overline{\partial_\nu u}dtdS\right) + C\|\ff\|^2_{H^0(\Omega_0)\times H^{-1}(\Omega_0)}.
\end{equation}
Notice this last inequality was obtained for $\ff$ with wavefront set in a conic neighborhood of a single covector. In order to generalize it to an arbitrary $\ff\in \mathcal{H}(\Omega_0)$ we use a microlocal partition of unity. 

Let's consider a finite pseudo-differential partition of unity $\{X_j\}_j$, whose symbols satisfy $1=\sum\chi_j$, thus, they localize in conic neighborhoods of a finite number of covectors $(x_j,\xi^j)\in WF(\ff)\cap T^*\overline{\Omega}_0$. Then, $\ff = (\text{Id}-\sum {\bf X}_j)\ff + \sum {\bf X}_j\ff$, where $WF(\ff)\cap WF(\text{Id}-\sum {\bf X}_j)=\emptyset$. We now set $\uu = e^{t{\bf P}_a}\ff$, the true solution to \eqref{PAT}, which from the previous satisfies $\uu\cong e^{t{\bf P}_a}\sum_j{\bf X}_j\ff = \sum_je^{t{\bf P}_a}{\bf X}_j\ff$. Let's denote $\uu^j = e^{t{\bf P}_a}{\bf X}_j\ff$. 
We have that inequality \eqref{ineq_f1} holds for each ${\bf X}_j\ff$, thus 
$$
\|{\bf X}_j\ff\|^2_{H^1(\Omega_0)\times H^0(\Omega_0)}\leq C\left(-\mathfrak{Re}\int_{(0,T)\times\partial \Omega}u^j_t\overline{\partial_\nu u^j}dtdS\right) + C\|{\bf X}_j\ff\|^2_{H^0(\Omega_0)\times H^{-1}(\Omega_0)}.
$$
Since $\uu^j$ solves \eqref{PAT}, the Robin boundary conditions imply
$$
\|{\bf X}_j\ff\|^2_{H^1(\Omega_0)\times H^0(\Omega_0)}\leq C\int_{(0,T)\times\Gamma}\lambda |u^j_t|^2dtdS + C\|{\bf X}_j\ff\|^2_{H^0(\Omega_0)\times H^{-1}(\Omega_0)}.
$$
Up to a smooth error, we can write $\uu^j = {\bf Q}{\bf X}_j\ff$ with ${\bf Q}$ the FIO (parametrix) constructed in Section \ref{subsec:GO}. By means of Egorov's theorem we can find another family of zero order $\Psi$DO's, namely $\{\tilde{\bf X}_j\}_j$, such that ${\bf Q}{\bf X}_j = \tilde{{\bf X}}_j{\bf Q}$ modulo smoothing operator, therefore we get $\uu^j\cong \tilde{{\bf X}}_j{\bf Q}\ff\cong \tilde{{\bf X}}_j e^{t{\bf P}_a}\ff=\tilde{{\bf X}}_j\uu$. Then,
$$
\begin{aligned}
\|\ff\|^2_{H^1(\Omega_0)\times H^0(\Omega_0)} &\leq \sum_j\|{\bf X}_j\ff\|^2_{H^1(\Omega_0)\times H^0(\Omega_0)} + \|(\text{Id}-\sum_j{\bf X}_j)\ff\|^2_{H^1(\Omega_0)\times H^0(\Omega_0)}\\
&\leq C\sum_j\int_{(0,T)\times\Gamma}\lambda |u^j_t|^2dtdS + C\|\ff\|^2_{H^0(\Omega_0)\times H^{-1}(\Omega_0)}\\
&\leq C\int_{(0,T)\times\Gamma}\lambda |u_t|^2dtdS + C\|\ff\|^2_{H^0(\Omega_0)\times H^{-1}(\Omega_0)}.
\end{aligned}
$$
The proof concludes by following the same compactness-uniqueness argument employed in previous sections, 
where we use that $\ff\mapsto \lambda^{1/2}u_t$ is a continuous and injective map (by Theorem \ref{thm:uniq_stab}) from $\cH(\Omega_0)$ to $L^2((0,T)\times\Gamma)$.
\end{proof}

\section{Numerical simulations with partial data}\label{sec:numerics}

The main purpose of the next numerical experiments is to illustrate the theoretical finding of previous sections, hence, issues related to optimality and accuracy of the implementations are out of the scope of this work. We attempt to compare the performance of the reconstruction in terms of $L^2$ and $L^\infty$ relative errors, between the back-projection approximation (i.e. solving system \eqref{Homan_TR}) and the dissipative Neumann Series approximation introduced above when only partial observations are available and the medium enforces a smoothly varying damping of acoustic waves. 

We consider a semi-bounded geometry where we set $\Omega = [-1,1]^2$ as our region of interest, and take an initial condition supported inside the subdomain $\Omega_0 = [-0.9667,0.9667]^2$. We assume that a portion of the boundary containing $x=-1$ is open and we assign Robin boundary condition to the rest of $\partial\Omega$ (see Figure \ref{fig:phantoms}).
The open region is simulated by considering a larger domain $\Omega'=[-1-\delta,1]\times[-1,1]$, for some $\delta>0$ specified below. 
In all of our simulations, we consider a $601\times601$ spatial grid inside $\Omega$ with mesh size $\Delta x = \Delta y = 0.0033$. In order to guarantee stability of the finite difference schemes employed in the simulations, we impose the Courant-Friedrichs-Lewy condition and set $\Delta t = 0.3\cdot\Delta x/(\sqrt{2}\max{c})$ where $c$ is a (known) sound speed.

The non-trapping sound speed is taken from \cite{QSUZ}, which is defined by the formula
\begin{figure}
\centerline{
\includegraphics[scale=0.2]{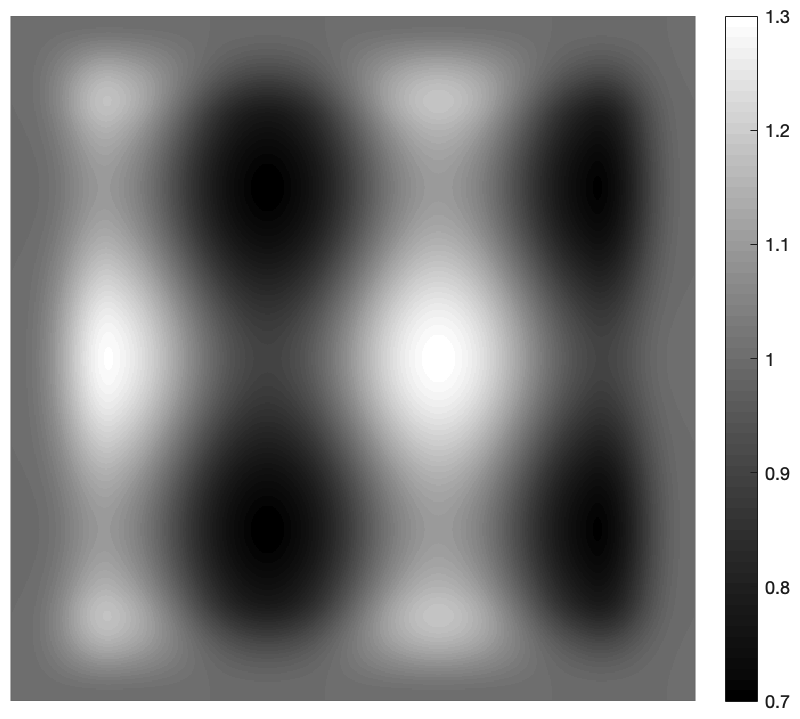}
}
\caption{Sound speed as in \eqref{ss}.}
\label{fig:ss}
\end{figure}
\begin{equation}\label{ss}
c(x,y) = \chi_1(x,y)\cdot\big(1.0 + 0.2\cdot\sin(2\pi x) + 0.1\cdot\cos(2\pi y)\big),
\end{equation}
with $\chi_1=1$ in $\Omega_0$, and smoothly bringing the sound speed to 1 near $\partial\Omega$. 
We work with two damping coefficients, namely, a soft linear attenuation and a stronger one proportional to the sound speed:
\begin{eqnarray}\label{a1}
a(x,y) &=& \chi_2(x,y)\cdot0.5\cdot(x+1);\\
\label{a2}
a(x,y) &=& \chi_2(x,y)\cdot2.0\cdot c(x,y).
\end{eqnarray}
In both cases we multiply by a smooth cutoff $\chi_2$, compactly supported in $\Omega$ in order to set a null damping near $\partial\Omega$ and outside $\Omega$. 

We perform two numerical simulations to test the convergence of the Neumann series. For the attenuation in \eqref{a1} we consider initial conditions $\ff = (f,-af)$ with $f$ the Shepp-Logan phantom, while for the second attenuation coefficient \eqref{a2} we consider $\ff$ of the same form but with $f$ given by two smaller copies of the Shepp-Logan phantom. In both cases, the initial source $f$ is slightly smoothed out to prevent numerical complications with large frequencies. 
\begin{figure}
\centerline{
\includegraphics[scale=0.2]{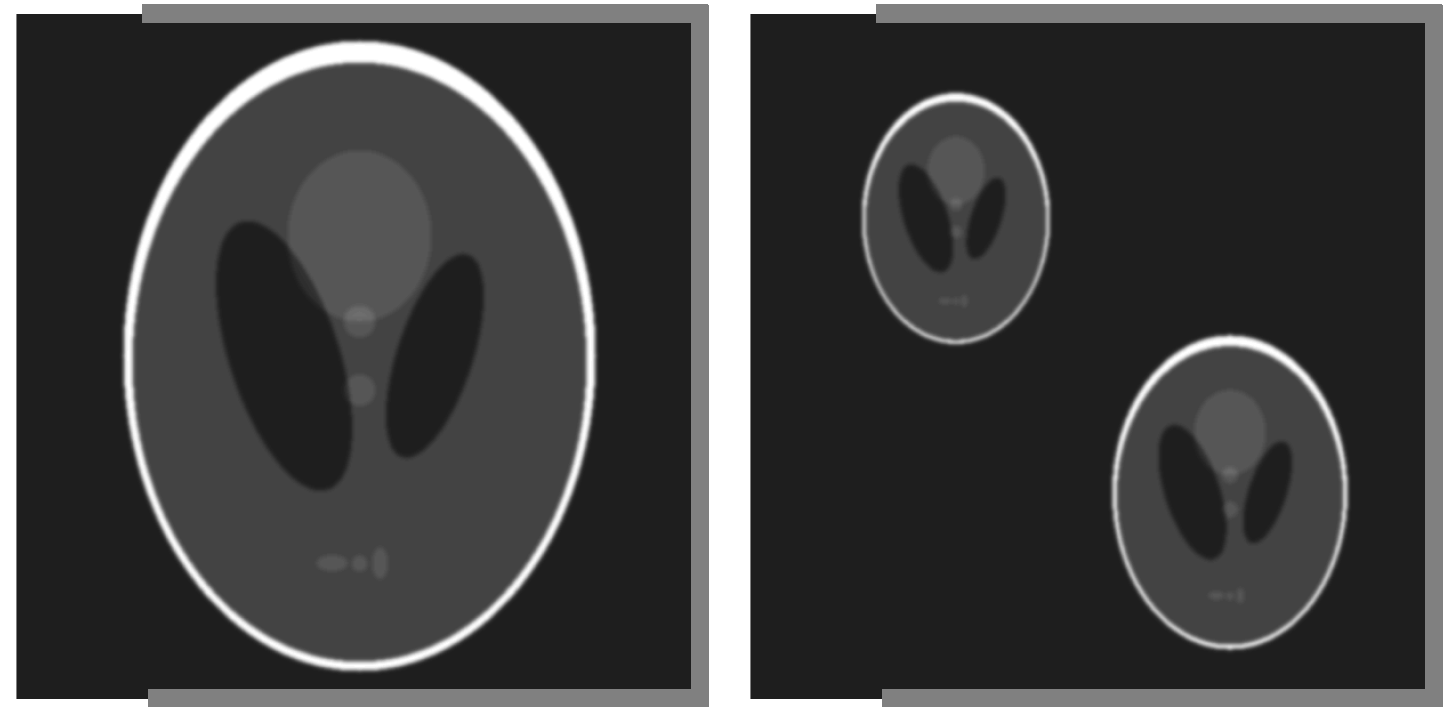}
}
\caption{Initial condition $f$ for simulation 1 ({\em left}) and 2 ({\em right}). The gray line surrounding the domain represents the support of $\lambda$, hence, the observation region.}
\label{fig:phantoms}
\end{figure}
Finally, the boundary observation region $\Gamma$ is a connected curved contained in the union of the edges $y=\pm 1$ and $x=1$ of the square domain $\Omega$, where we set an absorption coefficient $\lambda(x,y)=1$ in most of $\Gamma$ and decaying smoothly to zero as approaching $\partial\Gamma$. See Figure \ref{fig:phantoms}.

The back projection and the subsequent forward propagation in the Neumann series iterations were implemented by following the Perfectly Matched Layer (PML) scheme introduced in \cite{LTa}, which consist of a system of two first order equations with PML boundary conditions on the open part of $\partial\Omega'$ (with $\delta = 0.033$) to simulate an unbounded domain. The rest of the boundary, corresponding to $\Gamma$, is set with Robin conditions. The PML algorithm has been previously used in the context of PAT, for instance, in \cite{QSUZ} for the unattenuated case. 

With the purpose of avoiding the inverse crime, we simulate our data by using a standard finite difference scheme with fourth-order spatial accuracy and second-order accuracy in time. The reason behind the higher order spatial accuracy is to reduce the numerical dispersion of waves traveling across the domain. For simplicity we simulate the semi-bounded region by choosing a larger $\delta$ in the definition of $\Omega'$, and set a smoothly increasing attenuation to kill the propagation of waves away from $\Omega$. 

\begin{figure}
\centerline{
\includegraphics[scale=0.2]{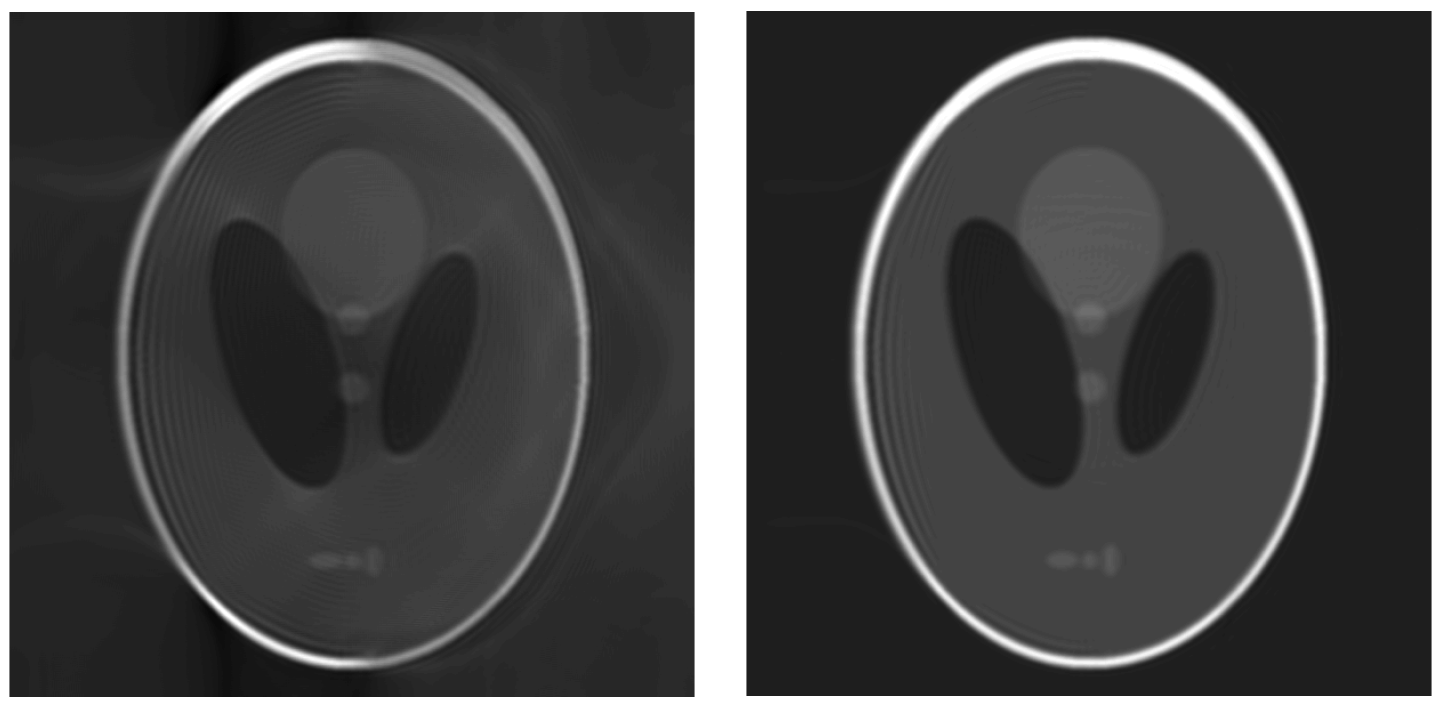}
}
\centerline{
\hspace{-0em}\includegraphics[scale=0.2]{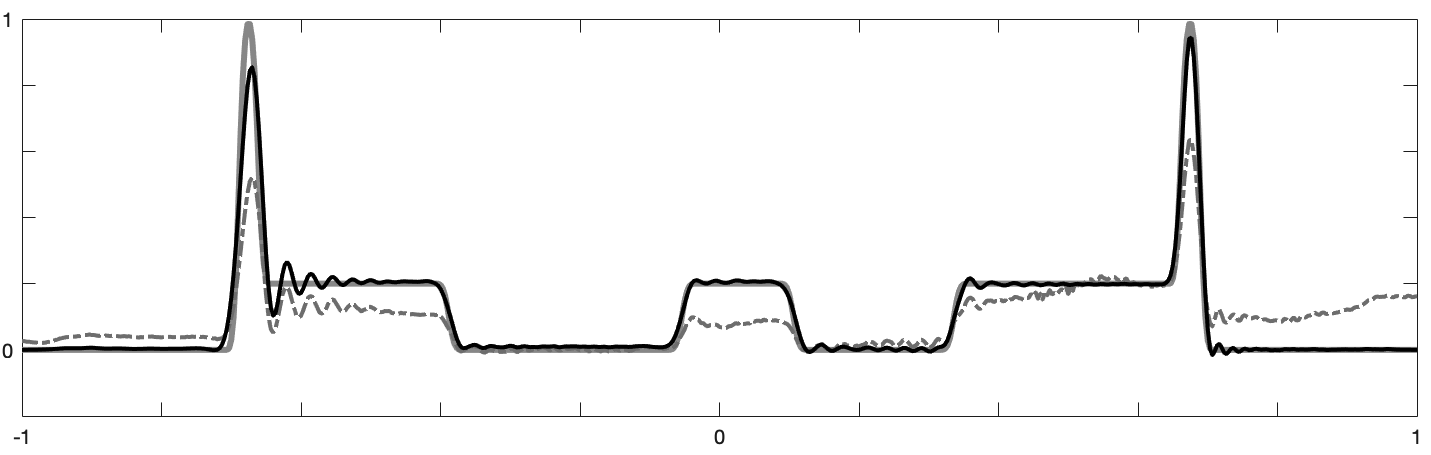}
}
\caption{Reconstruction of initial condition for the first simulation: back-projection ({\em top left}) and 20 terms of Neumann series ({\em top right}), both with pixel values on the interval $[-0.2,1]$. {\em Bottom:} cross section at $y=0$. The gray line, the black solid line and the gray dash-dotted line correspond, respectively, to the  true initial source, the Neumann series approximation and the back-projection approximations.}
\label{fig:rec_ex1}
\end{figure}

\begin{figure}
\centerline{
\includegraphics[scale=0.2]{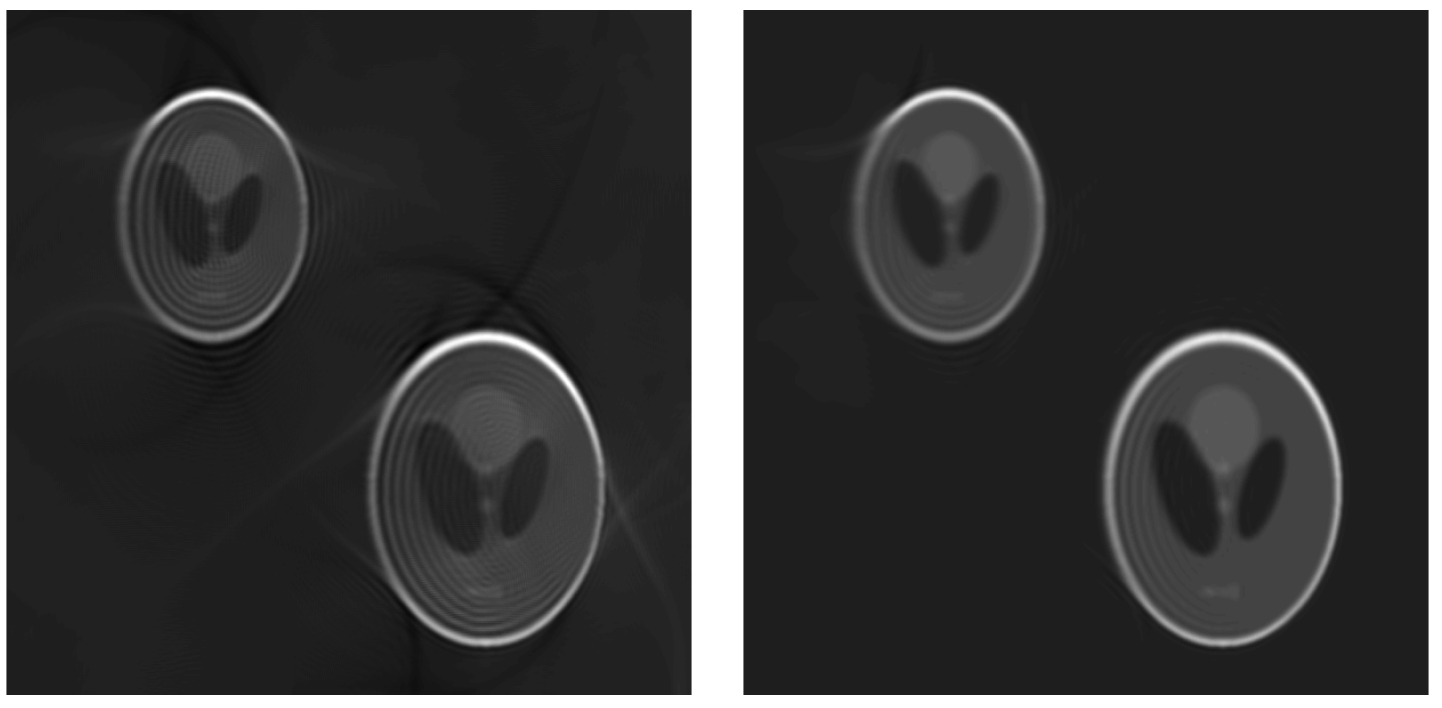}
}
\centerline{
\includegraphics[scale=0.2]{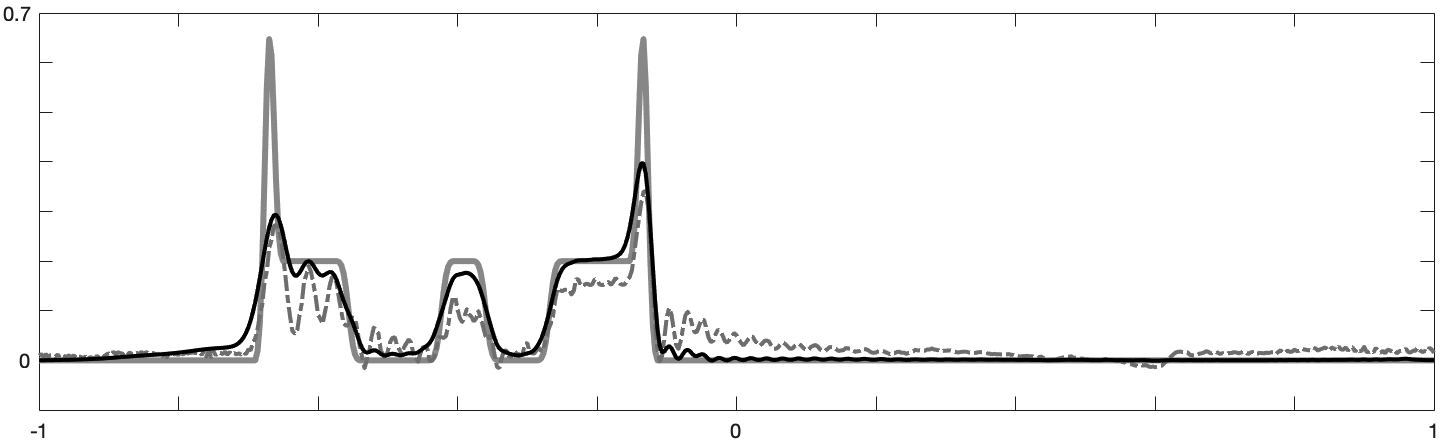}
}
\centerline{
\includegraphics[scale=0.2]{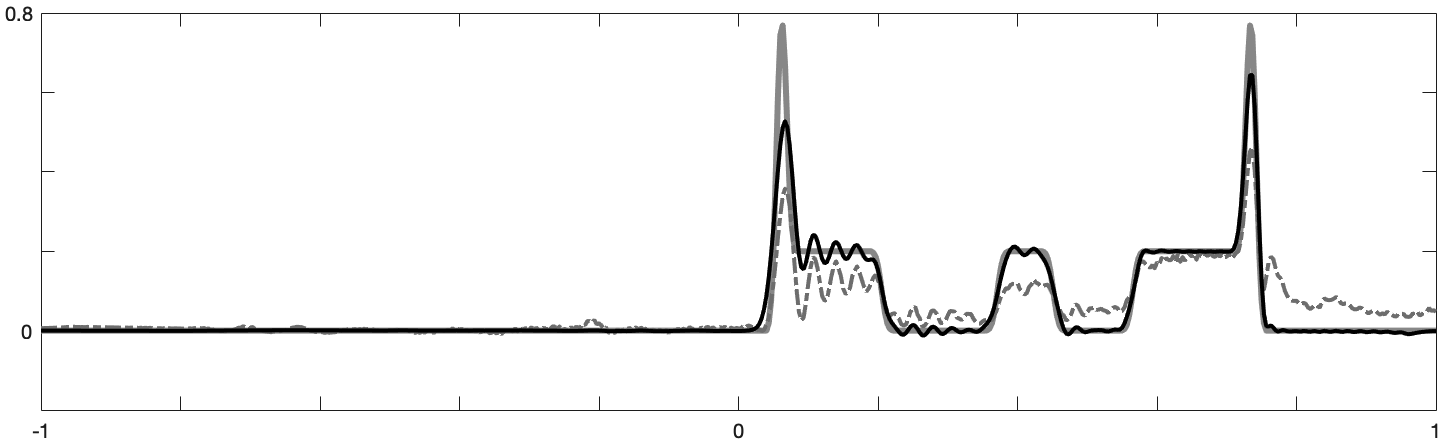}
}
\caption{Reconstruction of initial condition for the second simulation: back-projection ({\em left}) and 60 terms of the Neumann series ({\em right}), both with pixel values on the interval $[-0.2,1]$. {\em Middle and bottom:} cross section at $y=0.4$ and $y=-0.4$ respectively. The gray line, the black solid line and the gray dash-dotted line correspond, respectively, to the  true initial source, the Neumann series approximation and the back-projection approximations.}
\label{fig:rec_ex2}
\end{figure}

The first simulation, corresponding to the case of the soft linear attenuation in \eqref{a1} resulted in respective relative $L^\infty$ and $L^2$ errors of around 26\% and 9\% for the Neumann series approximation with 20 terms, compared to 53\% and 40\% for the standard back-projection. In the case of a strong attenuation as in \eqref{a2}, the relative $L^\infty$ and $L^2$ errors were respectively 49\% and 20\% for the Neumann series approximation with 60 terms, while 50\% and 32\% for the back-projection reconstruction. The results of these experiments can be visualized in Figures \ref{fig:rec_ex1} and \ref{fig:rec_ex2}, where top-view and cross-section images are presented to compare the performance of both reconstruction methods.

\section*{Acknowledgments}
Most of this work was done while the author was a W. H. Kruskal Instructor at the University of Chicago. The author is deeply grateful for the hospitality and support received at the University of Chicago and the Statistics Department, and particularly grateful for the support and mentoring of Guillaume Bal.

The author would also like to thanks Plamen Stefanov for valuable conversations and for bringing \cite{St} to the author's attention; and Sebastian Acosta and Carlos Montalto for their comments on early versions of the manuscript.
\bigskip


\end{document}